\documentclass[a4paper,11pt]{amsart}
\pdfoutput=1
\usepackage[utf8]{inputenc}

\usepackage[english]{babel}
\usepackage{csquotes}

\usepackage{amsmath}
\usepackage{amsfonts}
\usepackage{amssymb}
\usepackage{amsthm}
\usepackage{mathrsfs}
\usepackage{stmaryrd}
\usepackage[all]{xy}
\usepackage{multirow}

\usepackage[usenames,dvipsnames,svgnames,table]{xcolor}

\makeatletter
\renewcommand*{\@textcolor}[3]{
  \protect\leavevmode
  \begingroup
    \color#1{#2}#3
  \endgroup
}
\makeatother

\usepackage[colorlinks=false,bookmarks=true]{hyperref}
\hypersetup{linkbordercolor=cyan, linkcolor=cyan}
\hypersetup{citebordercolor=cyan, citecolor=cyan}

\newtheorem{thm}{Theorem}[section]

\newtheorem{lem}[thm]{Lemma}
\newtheorem{pps}[thm]{Proposition}

\newtheorem{thml}{Theorem}

\newtheorem{corl}[thml]{Corollary}

\newtheorem*{thm*}{Theorem}
\newtheorem*{cor*}{Corollary}
\newtheorem*{lem*}{Lemma}
\newtheorem*{pps*}{Proposition}

\theoremstyle{definition}
\newtheorem{dfn}[thm]{Definition}
\newtheorem{qst}[thm]{Question}
\newtheorem{pbl}[thm]{Problem}

\newtheorem{obs}[thm]{\emph{Remark}}
\newtheorem{exm}[thm]{\emph{Example}}

\newtheorem*{def*}{Definition}
\newtheorem*{def/thm*}{Definition{/}Theorem}

\newcommand{\N}{\mathbb{N}}
\newcommand{\Z}{\mathbb{Z}}
\newcommand{\Zum}{\mathbb{Z}_{\geq 1}}
\newcommand{\Nzero}{\mathbb{Z}_{\geq 0}}

\newcommand{\F}{\mathbb{F}}
\newcommand{\R}{\mathbb{R}}
\newcommand{\K}{\mathbb{K}}
\newcommand{\C}{\mathbb{C}}
\newcommand{\Mult}{\mathbb{G}_m}
\newcommand{\Addi}{\mathbb{G}_a}
\newcommand{\uCD}{\mathcal{G}_\Phi^{{\rm sc}}}
\newcommand{\ueCD}{E_\Phi^{{\rm sc}}}
\newcommand{\Lie}[1]{\mathfrak{#1}}

\newcommand{\laurent}[1]{#1[t,t^{-1}]}
\renewcommand{\H}{\mathcal{H}}
\renewcommand{\P}{\mathcal{P}}
\renewcommand{\L}{\mathcal{L}}
\newcommand{\LE}{\mathcal{LE}}
\newcommand{\U}{\mathcal{U}}
\newcommand{\mcB}{\mathcal{B}}
\newcommand{\B}{\mathcal{B}}
\newcommand{\Bzero}{\mathcal{B}_2^0(R)}
\newcommand{\G}{\mathcal{G}}
\newcommand{\OS}{\mathcal{O}_S}

\newcommand{\phee}{\varphi}
\newcommand{\veps}{\varepsilon}

\newcommand{\cupdot}{\mathbin{\mathaccent\cdot\cup}}
\newcommand{\tq}{~|~}
\newcommand{\tqalt}{~:~}

\newcommand{\til}[1]{\widetilde{#1}}
\newcommand{\set}[1]{\{#1\}}
\newcommand{\Eij}[1]{{\rm E}_{#1}}
\newcommand{\eij}[1]{e_{#1}}

\DeclareMathOperator{\Hom}{Hom}

\DeclareMathOperator{\rk}{rk}

\DeclareMathOperator{\GL}{GL}
\DeclareMathOperator{\SL}{SL}
\DeclareMathOperator{\Aut}{Aut}
\DeclareMathOperator{\ad}{ad}
\DeclareMathOperator{\vspan}{span}
\DeclareMathOperator{\Obj}{Obj}
\DeclareMathOperator{\carac}{char}
\DeclareMathOperator{\Adj}{Adj}
\DeclareMathOperator{\nonAdj}{NAdj}
\DeclareMathOperator{\Ext}{Ext}

\newcommand{\mbf}{\mathbf}
\newcommand{\mc}{\mathcal}
\newcommand{\mf}{\mathfrak}
\newcommand{\gera}[1]{\langle {#1} \rangle}
\newcommand{\vazio}{\varnothing}
\newcommand{\nsgp}{\trianglelefteq}

\newcommand{\NVB}{{\bf NVB}}
\newcommand{\NVBff}{{\bf NVB} for $\Phi$}
\newcommand{\QG}{{\bf QG}}
\newcommand{\QGff}{{\bf QG} for $\Phi$}

\newcommand{\bref}[1]{{\bf (\ref{#1})}}

\newcommand{\into}{\hookrightarrow}
\newcommand{\onto}{\twoheadrightarrow}

\usepackage[style=alphabetic,sorting=nyt,maxbibnames=7,maxcitenames=3,block=space,backend=bibtex]{biblatex}
\def\customdate{\empty}
\renewcommand{\date}[1]{\def\customdate{#1}}

\title[Presentations of parabolics in Chevalley groups]{Presentations of parabolics in some elementary Chevalley--Demazure groups}

\author[Y. Santos Rego]{Yuri Santos Rego}
\address{Fakult{\"a}t f{\"u}r Mathematik, Universit{\"a}t Bielefeld, Postfach 100131, D-33501, Deutschland}
\email{ysantos@math.uni-bielefeld.de}

\begin{document}

\thispagestyle{empty}

\begin{abstract}
 Given a universal elementary Chevalley--Demazure group $E_\Phi^{sc}(R)$ for which its (standard) parabolic subgroups are finitely generated, we consider the problem of classifying which parabolics $\P(R) \leq \ueCD(R)$ are finitely presented. We show that, under mild assumptions, this is equivalent to the finite presentability of a suitable retract of $\P$ which contains the Levi factor. If the base ring $R$ is a Dedekind domain of arithmetic type, we combine our results with well-known theorems due to Borel--Serre, Abels, Behr and Bux to give a partial classification of finitely presentable $S$-arithmetic subgroups of parabolics in split reductive linear algebraic groups.
\end{abstract}

\maketitle

\section{Introduction} \label{introdaintrodaintro}

\noindent

Let $R$ be a commutative ring with unity and let $\G_\Phi(R)$ be a Chevalley--Demazure group over $R$, where $\Phi$ is the associated (reduced) root system. Generators and relations for the elementary subgroup $E_\Phi(R) \leq \G_\Phi(R)$ and for related groups from algebraic $K$-theory were intensively studied in the last seven decades. Less understood are presentations of their subgroups, the main examples being the Borel subgroups in the cases where $R$ is a field (or more generally a semi-local ring) or a Dedekind domain of arithmetic type. Let $\P(R)$ be a non-trivial standard parabolic subgroup of the elementary subgroup $\ueCD(R)$ of a universal Chevalley--Demazure group $\uCD(R)$. In this work we investigate the following problem: Under which conditions can one assure that $\P(R)$ is finitely presented? Our main result, informally stated, is the following.

\begin{thml} \label{A0}
 Suppose the standard parabolic subgroups of $\ueCD(R)$ are finitely generated. Then, for ``most'' parabolic subgroups $\P(R) \leq \ueCD(R)$, there exists a subgroup $\LE(R) \leq \P(R)$ such that $\P(R)$ is finitely presented if and only if so is $\LE(R)$.
\end{thml}

\setcounter{thml}{0}

A precise statement of Theorem \ref{A0} will be given in Section \ref{main}. Before doing that, we elucidate the origins of the subgroup $\LE(R)$ above. Applications to the arithmetic case can be found in Section \ref{arithmetic}.

\subsection{Generalities} \label{introdaintro} Our starting point is the Levi decomposition $ \P(R) = \U(R) \rtimes \L(R) $, where $\U$ and $\L$ denote the unipotent radical and the Levi factor, respectively. Of course, if both $\U(R)$ and $\L(R)$ are finitely presented, then so is $\P(R)$. It is often the case, however, that $\U(R)$ is not even finitely generated. Since $\L(R)$ is a retract of $\P(R)$, finite presentability of the latter implies that of the former, and also that $\U(R) \nsgp \P(R)$ must be finitely generated as a normal subgroup. Furthermore, $\L(R)$ being finitely generated implies that $R$ is finitely generated as a ring. With those ingredients at hand, one might proceed to ask: What about going the other way around, that is, are those conditions sufficient?

Suppose an arbitrary semi-direct product $G = N \rtimes Q$ is given. Even assuming that $Q$ admits a presentation with some ``nice'' property (e.g. being finite or compact) and that $N$ is ``well-behaved'' with respect to the $Q$-action, there is no general characterization for when the same nice property holds for $G$. Even in particular cases, proving that $G$ has a ``nice'' presentation might be tricky (confer, for instance, the examples given in \cite{BBMS}). One then tipically varies the families of groups occurring in the short exact sequence $N \into G \onto Q$ to observe how (qualitative) properties of presentations of $G$ change with respect to $Q$ and $N$---a particularly successful case being that of metabelian groups \cite{BieriStrebel, LennoxRobinson}.

For the (split) short exact sequence $\U(R) \into \P(R) \onto \L(R)$, however, the expectations on the Levi factor $\L(R)$ are high: If the base ring $R$ is ``good enough'', then the action of $\L(R)$ on the unipotent radical $\U(R)$ is fairly well-understood, and some important structural and representation theoretical results hold (see e.g. \cite{AzadBarrySeitz, Stavrova09}). One might then expect some mild conditions under which a given nice presentation of $\L(R)$ can be enlarged to a nice presentation of $\P(R)$. Intuitively, the question is whether the Levi factor is ``too far'' from determining the whole parabolic $\P(R)$ and whether this well-understood action of $\L(R)$ on the unipotent radical $\U(R)$ is ``strong enough''. We aim to make this more precise.

\subsection{Detecting finite presentability of parabolics in \texorpdfstring{$\GL_n$}{GLn} -- an example} \label{exemplao} 
The simplest example of universal Chevalley--Demazure group scheme is the special linear group $\SL_n$, which coincides with its elementary subgroup $E_n$ if $n$ is sufficiently large and the base ring $R$ is e.g. Euclidean, semi-local or polynomial (or even Laurent polynomial) over a regular ring with vanishing special $K_1$ group (see, for instance, \cite{HahnO'Meara}).

Now, if $\mc{P}(R) \leq \SL_n(R)$ is a finitely generated parabolic, then via the determinant map we see that $\mc{P}(R)$ is finitely presented only if so is the corresponding parabolic subgroup $P(R)$ in the general linear group $\GL_n(R)$. This allows one to first look at examples of parabolic subgroups of some $\GL_n(R)$ as a test case.

Recall that a (standard) parabolic subgroup $P$ of $\GL_n(R)$ is a subgroup of generalized upper triangular matrices, that is, a subgroup of the following form.

\[
P =
\begin{pmatrix}
 \GL_{n_1} & * & * & \cdots & * \\
 \mbf{0} & \GL_{n_2} & * &  & \vdots \\
 \vdots & \ddots & \ddots & \ddots & \vdots \\
 \vdots & & \ddots & \ddots & * \\
 \mbf{0} & \cdots & & \mbf{0} & \GL_{n_k}
\end{pmatrix} \leq \GL_n(R),
\]
where $n_i \in \Zum$ for all $i$ and $n_1 + \ldots + n_k = n$. The \emph{trivial} parabolic subgroups in this case are $\GL_n(R)$ itself and the Borel subgroup $B_n(R) < \GL_n(R)$, i.e. the group of upper triangular matrices, for which $n_i = 1$ for all $i$. Thus, a non-trivial parabolic subgroup $P$ is constructed by starting with $B_n$ and enlarging some of the $1 \times 1$ blocks on the diagonal. Here, each occurrence of a $1 \times 1$ block on the diagonal corresponds to a copy of the multiplicative group of units $\GL_1(R) = R^\times$ of the base ring.

Let then $R$ equal $\laurent{\Z}$, the ring of integer Laurent polynomials. Here, Suslin's results apply \cite[Thm. 7.8 and Cor. 7.10]{SuslinSLnPoly} and so $\SL_n(R) = E_n(R)$ for $n \geq 3$. Now, consider the following parabolic subgroups of $\GL_{12}(\laurent{\Z})$.
\[
P_1 =
\begin{pmatrix}
 \GL_{1} & * & \cdots & * \\
 \mbf{0} & \GL_{5} & \ddots & \vdots \\
 \vdots & \ddots & \GL_{1} & * \\
 \mbf{0} & \cdots & \mbf{0} & \GL_{5}
\end{pmatrix}, ~ 
P_2 =
\begin{pmatrix}
 \GL_{5} & * & \cdots & * \\
 \mbf{0} & \GL_{1} & * & \vdots \\
 \vdots & 0 & \GL_{1} & * \\
 \mbf{0} & \cdots & \mbf{0} & \GL_{5}
\end{pmatrix}.
\]
Their Levi factors are, respectively, the subgroups
\[
L_1 =
\begin{pmatrix}
 \GL_{1} & \mbf{0} & \cdots & \mbf{0} \\
 \mbf{0} & \GL_{5} & \ddots & \vdots \\
 \vdots & \ddots & \GL_{1} & \mbf{0} \\
 \mbf{0} & \cdots & \mbf{0} & \GL_{5}
\end{pmatrix}, ~ 
L_2 =
\begin{pmatrix}
 \GL_{5} & \mbf{0} & \cdots & \mbf{0} \\
 \mbf{0} & \GL_{1} & 0 & \vdots \\
 \vdots & 0 & \GL_{1} & \mbf{0} \\
 \mbf{0} & \cdots & \mbf{0} & \GL_{5}
\end{pmatrix}.
\]
So both $L_1$ and $L_2$ are isomorphic to the direct product $\GL_1(\laurent{\Z})^2 \times \GL_5(\laurent{\Z})^2$.

We first observe that the necessary conditions for finite presentability of $P_1$ and $P_2$ are met. To begin with, the Laurent polynomials are additively generated by the set of powers $\set{t^n \tq n \in \Z}$. In other words, $\Z[t,t^{-1}]$ is generated, as a ring, by the singleton $\set{t}$. Since $\Z[t,t^{-1}]$ is a localization of the polynomial ring $\Z[t]$, it is a regular noetherian ring and its stable rank is at most 3 \cite[Thm. 4.1.11]{HahnO'Meara}, whence the unstable $K$-groups $K_{1,5}(\Z[t,t^{-1}])$ and $K_{2,5}(\Z[t,t^{-1}])$ are isomorphic to the $K$-groups $K_1(\Z[t,t^{-1}])$ and $K_2(\Z[t,t^{-1}])$, respectively \cite[Section 4.2E]{HahnO'Meara}. By Quillen's fundamental theorem \cite[Section 6, Corollary to Thm. 8]{KQuillen}, it then follows that
\begin{align*}
K_{1,5}(\laurent{\Z}) & \cong K_1(\Z) \oplus K_0(\Z) \mbox{ and} \\
K_{2,5}(\laurent{\Z}) & \cong K_2(\Z) \oplus K_1(\Z), 
\end{align*}
so the groups on the left hand side are both finitely generated (see e.g. \cite[p. 75]{Rosenberg} and \cite[Section 10]{KMilnor}). As a consequence, the $\GL_5(\laurent{\Z})$ blocks in the Levi factors $L_1$ and $L_2$ are finitely presented. Since $\GL_1(\laurent{\Z}) = \gera{\pm t} \cong \Z \rtimes C_2$, it follows that $L_1$ and $L_2$ are finitely presented. The unipotent radicals are given, respectively, by
\[
U_1 =
\begin{pmatrix}
 1 & * & \cdots & * \\
 \mbf{0} & \mbf{I}_5 & \ddots & \vdots \\
 \vdots & \ddots & 1 & * \\
 \mbf{0} & \cdots & \mbf{0} & \mbf{I}_5
\end{pmatrix}, ~ 
U_2 =
\begin{pmatrix}
 \mbf{I}_5 & * & \cdots & * \\
 \mbf{0} & 1 & * & \vdots \\
 \vdots & 0 & 1 & * \\
 \mbf{0} & \cdots & \mbf{0} & \mbf{I}_5
\end{pmatrix},
\]
so they are nilpotent groups whose factors are direct sums of the underlying additive group $(\laurent{\Z},+)$. Since $L_1$ and $L_2$ contain the diagonal subgroup of $\GL_{12}(\laurent{\Z})$, it follows from the conjugation action of diagonal matrices that $U_1$ and $U_2$ are finitely generated as normal subgroups of $P_1$ and $P_2$, respectively.

So far, so good. Let us now look at $P_1$. Since $U_1$ is nilpotent, one can make use of its descending central series to construct a presentation for it whose ``most important'' relations---besides the ones induced by the additive relations from $\laurent{\Z}$---are just the necessary commutator relations between elementary matrices (see Section \ref{Levi}). However, these are necessarily infinite in number as the additive group $(\laurent{\Z},+)$ is infinitely generated. By making use of the diagonal matrices, we shall see in Section \ref{teoremao} how to reduce those commutator relations to just finitely many. Furthermore, we will see how to push the remaining additive relations induced by $\laurent{\Z}$ from $U_1$ to $L_1$ so that such relations are actually consequences of analogue relations found in $L_1$. It will then follow that $P_1$ is in fact finitely presented, as our intuition on the ``strong'' action of $L_1$ on $U_1$ might have predicted.

However, even though the Levi factors $L_1$ and $L_2$ are isomorphic, the parabolic subgroup $P_2$ is \emph{not} finitely presented. In fact, a map sending the $\GL_5$ blocks and most of the unipotent radical $U_2$ to the identity induces a \emph{retraction}, depicted below, of $P_2$ onto the Borel subgroup of $\GL_2(\laurent{\Z})$.
\[
\begin{pmatrix}
 \GL_{5} & * & \cdots & * \\
 \mbf{0} & \GL_{1} & * & \vdots \\
 \vdots & 0 & \GL_{1} & * \\
 \mbf{0} & \cdots & \mbf{0} & \GL_{5}
\end{pmatrix}
\onto
\begin{pmatrix}
 \mbf{I}_5 & \mbf{0} & \cdots & \mbf{0} \\
 \mbf{0} & \GL_{1} & * & \vdots \\
 \vdots & 0 & \GL_{1} & \mbf{0} \\
 \mbf{0} & \cdots & \mbf{0} & \mbf{I}_5
\end{pmatrix}
\cong
\begin{pmatrix}
 * & * \\
 0 & *
\end{pmatrix}
\]
Such Borel subgroup clearly contains the matrices 
\[
\begin{pmatrix} t & 0 \\ 0 & t^{-1} \end{pmatrix} \mbox{ and } \begin{pmatrix} 1 & 1 \\ 0 & 1 \end{pmatrix},
\]
so it cannot be finitely presented by a result of Krsti\'c--McCool \cite[Section 4]{KrsticMcCool}. Therefore $P_2$ itself cannot be finitely presented.

This example suggests the following. Though the Levi factor alone might fail to detect whether the whole parabolic is finitely presentable or not, one could remedy the situation by slightly enlarging it as follows and then ask if the resulting group encodes the desired information. Let us call \emph{\textcolor{blue}{regular}} the $\GL_{n_i}$ blocks of the Levi factor for which $n_i \geq 2$. For every sequence of \emph{at least} two $\GL_1$ blocks in a row, define a \emph{\textcolor{red}{triangular} block} to be the subgroup generated by the elementary matrices that occur right above those $\GL_1$ blocks. In this set-up, we define the \emph{extended Levi factor} of the given parabolic to be its subgroup generated by the diagonal matrices and by both its regular and its triangular blocks. For $P_1$ and $P_2$, the respective extended Levi factors $LE_1$ and $LE_2$ are depicted below.
\[
P_1 =
\begin{pmatrix}
 \GL_{1} & * & \cdots & * \\
 \mbf{0} & \GL_{5} & \ddots & \vdots \\
 \vdots & \ddots & \GL_{1} & * \\
 \mbf{0} & \cdots & \mbf{0} & \GL_{5}
\end{pmatrix}
\geq 
\begin{pmatrix}
 \GL_{1} & \multicolumn{1}{|c}{\mbf{0}} & \cdots & \mbf{0} \\ \cline{1-2}
 \mbf{0} & \multicolumn{1}{|c|}{\textcolor{blue}{\GL_{5}}} & \ddots & \vdots \\ \cline{2-3}
 \vdots & \ddots & \multicolumn{1}{|c|}{\GL_{1}} & \mbf{0} \\ \cline{3-4}
 \mbf{0} & \cdots & \multicolumn{1}{c|}{\mbf{0}} & \textcolor{blue}{\GL_{5}}
\end{pmatrix} = LE_1,
\]

\[
P_2 = 
\begin{pmatrix}
 \GL_{5} & * & \cdots & * \\
 \mbf{0} & \GL_{1} & * & \vdots \\
 \vdots & 0 & \GL_{1} & * \\
 \mbf{0} & \cdots & \mbf{0} & \GL_{5}
\end{pmatrix}
\geq
\begin{pmatrix}
 \textcolor{blue}{\GL_{5}} & \multicolumn{1}{|c}{\mbf{0}} & \cdots & \mbf{0} \\ \cline{1-3}
 \mbf{0} & \multicolumn{1}{|c}{\GL_{1}} & \multicolumn{1}{c|}{\textcolor{red}{*}} & \vdots \\ 
 \vdots & \multicolumn{1}{|c}{\textcolor{red}{0}} & \multicolumn{1}{c|}{\GL_{1}} & \mbf{0} \\ \cline{2-4}
 \mbf{0} & \cdots & \multicolumn{1}{c|}{\mbf{0}} & \textcolor{blue}{\GL_{5}}
\end{pmatrix} = LE_2.
\]
We see that $LE_1$ coincides with the Levi factor $L_1$ because $P_1$ contains no triangular blocks, whereas $L_2$ is a proper subgroup of $LE_2$ and the latter contains the obstructive Borel subgroup described above. In particular, $LE_1$ is finitely presented (as well as $P_1$), but $LE_2$, and so $P_2$, are not.

\subsection{Main results} \label{main} We extend the ideas presented in Section \ref{exemplao} to universal elementary Chevalley--Demazure groups and confirm that, in many cases, the extended Levi factor indeed gives a characterization of finite presentability of parabolics. In this paper we closely follow Steinberg's notation for the Chevalley--Demazure groups \cite{Steinberg}.

Here, $\Phi$ denotes a reduced, irreducible root system (see e.g. \cite[Chap. 6]{BourbakiLie4-6} or \cite[Chap. 3]{Humphreys}). Fix a set $\Delta \subseteq \Phi$ of simple roots and an arbitrary total order on $\Delta$ compatible with height of roots. We then speak of standard parabolic subgroups relative to this arbitrary, but fixed, ordered $\Delta$. We also do not distinguish between $\Delta$ and the corresponding Dynkin diagram, so topological (or combinatorial) properties of simple roots, such as adjacency, are interpreted as properties of vertices in the Dynkin diagram.

Recall that the elementary subgroup $\ueCD(R)$ is the subgroup of $\uCD(R)$ generated by all the unipotent root subgroups $\mf{X}_\alpha = \langle \{ x_\alpha(r) \tq \alpha \in \Delta,~ r \in R \} \rangle \leq  \uCD(R)$. For $\SL_n$, i.e. the case $\Phi = A_{n-1}$, the group $E_{A_{n-1}}^{sc}(R) =: E_n(R)$ is just the subgroup generated by elementary matrices. Given $I \subseteq \Delta$, denote by $\P_I(R) \leq \ueCD(R)$ the parabolic subgroup associated to it (see Section \ref{preliminar} for details).

\begin{def*}
 Let $\P_I(R) \leq \ueCD(R)$ be a (standard) parabolic subgroup. If $I \neq \vazio$, its \emph{extended Levi factor}, denoted $\LE_I(R)$, is given by
\[
\LE_I(R) := \gera{\L_I(R), \mf{X}_\alpha \tqalt \alpha \in \Delta \backslash I \mbox{ is not adjacent to any element of } I}.
\]
If $I = \vazio$, then the $n=\rk(\Phi)$ extended Levi factors of $\P_\vazio(R)$ are given by
\[
\LE_i(R) := \gera{\L_\vazio(R), \mf{X}_{\alpha_i} \tqalt \alpha_i \mbox{ is the } i\mbox{-th root of } \Delta}.
\]
\end{def*}

In the notation of Section \ref{exemplao}, for the case $I \neq \vazio$ the root subgroups $\mf{X}_\alpha$ with $\alpha \in \Delta \backslash I$ span the triangular blocks of $\P_I(R)$. The Levi factor $\L_I(R)$ is generated by both the regular blocks and the standard torus, which in the Chevalley--Demazure setting plays the role of the subgroup of diagonal matrices. The main result of this paper is the following.

\begin{thml}[restated] \label{A}
Let $\ueCD(R)$ be a universal elementary Chevalley--Demazure group for which its (standard) parabolic subgroups are finitely generated. If the base ring $R$ is \QGff, then a standard parabolic subgroup $\P_I(R) \leq \ueCD(R)$ is finitely presented if and only if its extended Levi factors are finitely presented, except possibly in the case where $I = \set{\alpha}$ with $\alpha$ a long root in the root system of type $G_2$.
\end{thml}

The proof is given in Section \ref{teoremao} using generators and relators \`a la Steinberg, and is elementary in the sense that it heavily relies only on so-called ``elementary calculations'', i.e. commutator calculus paired with the Chevalley commutator formula for root subgroups. 

Theorem \ref{A} still needs some explanation. Denote by $\Bzero \leq E_{A_1}^{sc}(R)$ the standard Borel subgroup in type $A_1$. We say that $R$ is ``quite good'' for the root system $\Phi$---or {\QG} for short---if $\Bzero$ is finitely presented \emph{or} $R$ is ``not very bad'' for $\Phi$ (abbreviated \NVB), that is to say
\[
\begin{cases}
 2 \in R^\times, & \mbox{if } \Phi \in \set{B_n, C_n, F_4};\\
 2, 3 \in R^\times, & \mbox{if } \Phi = G_2.
\end{cases}
\]
The {\NVB} condition is a common assumption when dealing with elementary calculations (see, for instance, \cite{Stavrova09, AzadBarrySeitz, VavilovParabs, SuzukiParabs, Stein}). The point is that both the finite presentability of $\Bzero$ and the {\NVB} condition allow one to overcome the technicalities with structure constants that appear in the commutators. Nevertheless, we strongly suspect that the equivalence in Theorem \ref{A} holds for arbitrary (finitely generated) commutative rings (with $1$). The exceptional case $\P_{\set{\alpha}}(R) \leq E_{G_2}^{sc}(R)$ above---with $R$ being {\NVBff} but $\Bzero$ \emph{not} finitely presented---is the only instance where our computations with the chosen defining relators were inconclusive, and the problem there can be made quite explicit; see Remark \ref{precisao} for details. The term ``most parabolics'' from the previous formulation (Section \ref{introdaintrodaintro}) thus means that the possible exceptions for Theorem \ref{A} are the parabolic subgroup $\P_{\set{\alpha}}(R) \leq E_{G_2}^{sc}(R)$, with $\alpha$ long and $\Bzero$ \emph{not} finitely presented, and the parabolics $\P_I(R) \leq E_{\Phi}^{sc}(R)$ for $R$ \emph{not} \QGff.

For ease of reference, we state below a special case of Theorem \ref{A} as a corollary. Recall that a standard parabolic subgroup $\P_I(R) \leq \ueCD(R)$ is \emph{maximal} if the only standard parabolic subgroup properly containing it is the whole elementary group $\ueCD(R)$. In particular, every root in $\Delta$ is adjacent to $I$ in this case, so we get $\LE_I(R) = \L_I(R)$ whenever $\P_I(R)$ is maximal. Now, if $\Phi$ is simply-laced, that is, if all its roots have the same length, then every ring is {\QGff} because the {\NVB} condition imposes no restriction on such root systems. We thus obtain the following.

\begin{corl} \label{classifsimplylaced}
 Suppose that all (standard) parabolic subgroups of $\ueCD(R)$ are finitely generated and that the root system $\Phi$ is simply-laced of rank at least 2. Then a standard, non-trivial, maximal parabolic subgroup of $\ueCD(R)$ is finitely presented if and only if its Levi factor is finitely presented.
\end{corl}

As stated, the proof of Theorem \ref{A} is done via elementary calculations, so similar methods occur in many places in the literature, with different applications in mind. For instance, there has been a great deal of work on the structure of normal subgroups of linear groups related to elementary subgroups (cf. \cite{BassMilnorSerre, Matsumoto, Abe69, AbeSuzuki, Abe89, Vaserstein, Stavrova09}). Generators and relators themselves and low-dimensional $K$-groups have been investigated e.g. in \cite{Steinberg0, Cohn, O'Meara, Silvester, Stein, SteinK1K2, RehmannSoule, HahnO'Meara, KrsticMcCool} as well as in \cite{Siegfried}, where, it seems, the computations most closely resemble the ones we do here. More recently, similar calculations have also been employed in the Kac--Moody setting, starting right from Tits' presentation in \cite{TitsKM}, and then e.g. in \cite{Allcock16St, CapLuRe, Allcock16KM, AllcockCarbone}. It is likely that results analogous to ours can be extended to some subgroups of Kac--Moody groups.

\subsection{Motivation and application -- the arithmetic case} \label{arithmetic} There has been a long quest to understand finitely presentable $S$-arithmetic groups. The theory took a serious turn when Nagao showed in \cite{Nagao}, in particular, that $\SL_2(\F_q[t])$ is not even finitely generated. Such a phenomenon did not seem to occur in characteristic zero, and after important developments and examples established by many mathematicians, A. Borel and J.-P. Serre \cite{BoSe} proved that any $S$-arithmetic subgroup of a reductive group over an algebraic number field is finitely presented.

In contrast, the discoveries in the function field case pointed out to a dependency on both the rank of the underlying reductive group and on the number of places $|S|$ to achieve finite presentability. A full characterization was finally given by H. Behr in \cite{Behr98} when he established the 2-dimensional rank theorem.

Until the early 80's, the theory was in poor shape for arbitrary algebraic subgroups of reductive (or even semi-simple) groups. In a remarkable work, H. Abels completely classified all finitely presentable $S$-arithmetic subgroups in characteristic zero by reducing the problem to the soluble case and establishing necessary and sufficient conditions for finite presentability there \cite{Abels}. He showed, in particular, that an $S$-arithmetic Borel subgroup is always finitely presented.

Again, the case was different in positive characteristic \cite{Bux0}, and K.-U. Bux later proved, in particular, that such $S$-arithmetic Borel subgroups are finitely presented if and only if $|S| \geq 3$ \cite{Bux04}. The main motivation for this work was precisely the natural follow-up question to those classification results: What happens in between, i.e. what about $S$-arithmetic subgroups of parabolic subgroups of reductive groups? We apply Theorem \ref{A} to obtain a partial classification of finite presentability in this case.

We say that a linear algebraic group $\mathbf{H}$, defined over a field $k$, retracts onto an \emph{almost Borel group} if there exists a $k$-retraction $r : \mbf{H} \onto \Addi \rtimes \mbf{T}$ onto a split, connected, soluble $k$-group $\Addi \rtimes \mbf{T}$ such that $\mathbf{T}$ is a $k$-split subtorus of $\mbf{H}$ of rank at least 1 acting non-trivially on $\Addi$. We combine our results with the well-known theorems of Borel--Serre, Abels, Behr and Bux mentioned above to, on the one hand, recover a familiar fact in the number field case and, on the other hand, establish finite presentability of $S$-arithmetic groups in new cases.

\begin{thml} \label{B}
Let $\mbf{G}$ be a split, connected, reductive, linear algebraic group defined over a global field $\K$ and let $S \neq \vazio$ be a finite set of places of $\K$ containing all the archimedean ones. Suppose $|S| > 1$ if $\K$ is a global function field and let $\mbf{P} \leq \mbf{G}$ be an arbitrary proper parabolic subgroup (possibly a Borel subgroup).
\begin{enumerate}
 \item \label{B.1} If $\carac(\K) = 0$, then the $S$-arithmetic subgroups of $\mbf{P}$ are always finitely presented;
 \item Assume $\carac(\K) > 0$.
  \begin{enumerate}
   \item \label{B.2a} If $\mbf{P}$ retracts onto an almost Borel group, then its $S$-arithmetic subgroups are finitely presented if and only if $|S| \geq 3$;
   \item \label{B.2b} Otherwise, and if $\K$ is {\bf NVB} for the underlying root system of $\mbf{G}$, then an $S$-arithmetic subgroup $\Gamma \leq \mbf{P}$ is finitely presented if and only if an $S$-arithmetic subgroup $\Lambda$ of the Levi factor $\mbf{L} \leq \mbf{P}$ is finitely presented.
  \end{enumerate}
\end{enumerate}
\end{thml}

The proof can be found in Section \ref{aplicacoes} and is a straightforward application of Theorem \ref{A}, the theorems of Borel--Serre, Behr, Abels and Bux, and standard arguments from the theory of arithmetic groups. It should be stressed that Theorem \ref{B} also holds for the exceptional parabolic in type $G_2$ which was excluded from Theorem \ref{A}.

As mentioned, Part \bref{B.1} is not new. Though not formally stated in \cite{Abels}, it originally follows from \cite[Thms. 5.6.1 and 6.2.3]{Abels} together with Kneser's local-global principle \cite{Kneser} and Borel--Tits' compactness theorem \cite[Prop. 9.3]{BorelTits}, or more swiftly as a special case of Tiemeyer's results \cite[Cor. 4.5]{Tiemeyer}, which also rely on \cite[Prop. 9.3]{BorelTits}. The use of Theorem \ref{A} here and Abels' strategy, however, point out to an alternative proof of this that would not depend on \cite[Prop. 9.3]{BorelTits}: Since the Borel subgroups of $\SL_2$ are metabelian, one could mimic the arguments from \cite[Chp. 7]{Abels} in order to apply classical Bieri--Strebel theory \cite{BieriStrebel} directly to the $S$-arithmetic subgroups of such groups and conclude that they are always finitely presented. Pairing this with \cite[Thm. 6.2]{BoSe}, the claim would follow from Theorem \ref{A}. 

Part \bref{B.2a} contains some new cases and includes, in particular, the above mentioned theorem of Bux on the finite presentability of $S$-arithmetic Borel groups in positive characteristic. For the proof we outline two different approaches, one of which is independent of \cite{Bux04}; see Section \ref{aplicacoes} for details. 

The stated results for non-minimal proper parabolics over function fields were, to the best of our knowledge, unknown. They make use of Behr's rank  theorem \cite{Behr98}. Applying it more explicitly, we can make the characterization of the {\NVB} case in \bref{B.2b} more precise: The semi-simple part of $\mbf{L}$ is covered by a direct product of universal Chevalley--Demazure groups $\G_{\Phi_i}^{sc}$, each of which has global rank $d_i = |S| \cdot \rk(\Phi_i)$; putting $d := \mbox{min}_{i}\{d_i\}$, it follows from \cite{Behr98} that $\Gamma$ is finitely presented if and only if $d \geq 3$.

\subsection{Structure of the paper} In the preliminary Section \ref{jargao} we recall the construction of Chevalley--Demazure group schemes and some key properties. The definition of parabolic subgroups as well as some properties to be used in the sequel are given in Section \ref{parabolicos}. For convenience, we summarize in \ref{notacao} the notation to be used throughout. Section \ref{Levi} recalls retraction arguments for presentations of group extensions and introduces the extended Levi factor. We then deal with generators and relators for the unipotent radical and related subgroups. The main result, Theorem \ref{A}, is proved in Section \ref{teoremao}. The following Section \ref{aplicacoes} is devoted to the case of algebraic groups and their $S$-arithmetic subgroups, including the proof of Theorem \ref{B} on finite presentability of $S$-arithmetic parabolics. We finish the paper with some remarks and questions in Section \ref{remarks}.

{\small
\subsection*{Acknowledgments} I thank my advisor, Prof. Kai-Uwe Bux, for his patience and guidance, and for introducing me to this topic. I am indebted to Prof. Herbert Abels for many lively mathematical sessions. Thanks to Stefan Witzel, Dawid Kielak and Alastair Litterick for many helpful conversations and to Paula Macedo Lins de Araujo for her support and her help on a previous version of this paper. This work was supported by the Deutscher Akadamischer Austauschdienst and the Sonderforschungsbereich 701 in Bielefeld, and is part of the author's PhD project.
}

\section{Preliminaries} \label{preliminar}

\subsection{Chevalley--Demazure group schemes} \label{jargao} Chevalley groups play a paramount role in the theories of algebraic groups and finite simple groups and have been intensively studied in the last six decades. Roughly speaking, a Chevalley--Demazure group scheme is a representable functor (cf. \cite{Waterhouse, Kostant}) from the category of commutative rings to the category of groups which is, in some sense, uniquely associated to a complex, connected, semi-simple Lie group and to a certain lattice of weights of the corresponding Lie algebra. We recall the general construction of Chevalley--Demazure group schemes over $\Z$ along the lines of \cite{Abe69} and \cite{Kostant} and state below a slightly more precise definition of such functors. The material presented here is standard and we refer e.g. to \cite{Chevalley, Ree, Kostant, Steinberg, Abe69, BorelSem68, SGA3.3} for details, so the familiar reader might prefer to skip most of this section and refer back to Section \ref{notacao} for notation, if needed. 

Let $G_\C$ be a complex, connected, semi-simple Lie group and let $\Lie{g}$ its Lie algebra with a Cartan subalgebra $\Lie{h} \subseteq \Lie{g}$ and associated (reduced) root system $\Phi \subseteq \Lie{h}^*$. In his seminal Tohoku paper, C. Chevalley established the following.

\begin{thm*}[\cite{ChevalleyTohoku}] \label{Chevalleybasis}
 There exist non-zero elements $X_\alpha\in \Lie{g}$, where $\alpha$ runs over $\Phi$, with the following properties:
 \begin{enumerate}
  \item Given $\alpha, \beta \in \Phi$ with $\alpha \neq -\beta$, if $\alpha+\beta \in \Phi$, then $[X_\alpha,X_\beta] = \pm (m+1)X_{\alpha+\beta}$, where $m$ is the largest integer for which $\beta-m\alpha\in\Phi$; otherwise $[X_\alpha,X_\beta] = 0$;
  \item $X_\alpha \in \Lie{g}_\alpha = \{X \in \Lie{g} \tq \ad(H)X = \alpha(H)X ~\forall H \in \Lie{h} \}$, i.e. each vector $X_\alpha$ belongs to the weight space of $\Lie{h}$ under the adjoint representation;
  \item Setting $H_\alpha = [X_\alpha, X_\alpha]$ and $(\alpha,\beta) := 2\frac{\gera{\alpha,\beta}}{\gera{\beta,\beta}} \in \Z$ for $\alpha, \beta \in \Phi$, one has that $H_\alpha \in \Lie{h} \backslash \set{0}$ and $[H_\alpha, X_\beta] = (\beta, \alpha) X_\beta$;
  \item $\set{H_\alpha}_{\alpha \in \Phi}$ spans $\Lie{h}$, the set $\set{H_\alpha, X_\alpha}_{\alpha \in \Phi}$ is a basis for $\Lie{g}$, and there is a decomposition $\Lie{g} = \Lie{h} \oplus \left( \oplus_{\alpha \in \Phi} \Lie{g}_\alpha \right)$.
 \end{enumerate}
\end{thm*}

A basis $\set{H_\alpha, X_\alpha}_{\alpha \in \Phi}$ for $\Lie{g}$ as above is known as a Chevalley basis and the $\Z$-Lie ring $\Lie{g}_\Z$ generated by it is sometimes called a Chevalley lattice. Using $\Lie{g}_\Z$ alone, one may already proceed to construct the first examples of Chevalley--Demazure groups over fields, namely those of adjoint type, which in fact yield many infinite families of simple groups. These were the ones introduced in \cite{ChevalleyTohoku}, popularized as Chevalley groups (see e.g. \cite{Carter}). The next step to construct more general group schemes is to allow for different representations of $\Lie{g}$.

Let $P_{sc} = \set{ \chi \in \Lie{h} \tq \chi(H) \in \Z ~\forall H \in \Lie{h}}$ be the lattice of weights of $\Lie{h}$ and let $P_{ad} = \vspan_\Z(\Phi) \subseteq P_{sc}$ be the root lattice. If $\rho: \Lie{g} \to \Lie{gl}(V)$ is a faithful representation of $\Lie{g}$, then $P_{ad} \subseteq P_\rho \subseteq P_{sc}$, where $P_\rho = \set{ \chi \in \Lie{h}^* \tq V_\chi \neq \set{0}}$ denotes the lattice of weights of the representation $\rho$ (recall that $V_\chi = \set{X \in V \tq \chi(H)X = \rho(H)X ~\forall H \in \Lie{h}}$). Conversely, given $P \subseteq \Lie{h}^*$ with $P_{ad} \subseteq P \subseteq P_{sc}$, there exists a faithful representation $\rho : \Lie{g} \to \Lie{gl}(V)$ such that $P_\rho = P$ (see, for instance, \cite{Ree, BorelSem68, Jacobson}).

Fix a lattice $P := P_\rho$ as above. From Kostant's construction \cite[Thm. 1 and Cor. 1]{Kostant}, one can define a $\Z$-lattice $B_\rho$ of the universal enveloping algebra $U(\Lie{g})$ and a certain family $F$ of ideals of $B_\rho$ \cite[Section 1.3 and p. 98]{Kostant} such that
$$\Z[G_\C,P] := \set{f \in \Hom(B_\rho, \Z) \tq f \mbox{ vanishes on some } I \in F}$$
is a Hopf algebra over $\Z$ and the following hold:
\begin{enumerate}
 \item $\Z[G_\C,P]$ is a finitely generated integral domain;
 \item \label{unico} The coordinate ring $\C[G_\C]$ is isomorphic to the Hopf algebra $\Z[G_\C,P] \otimes_\Z \C$.
\end{enumerate}
In particular, we get a representable functor $\G_\Phi^P := \Hom_\Z(\Z[G_\C,P],-)$ from the category of commutative rings with $1$ to the category of groups. Since the Lie group $G_\C$ and the representation $\rho$ are determined, up to isomorphism, by the root system $\Phi$ and the lattice $P$, respectively, we see that $\G_\Phi^P$ in fact depends only on $\Phi$ and $P$ up to isomorphism. Moreover, by property \bref{unico} above we recover $G_\C \cong \G_\Phi^P(\C)$ as the group of $\C$-points of $\G_\Phi^P$. The functor $\G_\Phi^P$ also inherits some properties of $G_\C$, namely, it is semi-simple (in the sense of \cite{SGA3.3}) and contains a maximal torus of rank $\rk(\Phi)$ defined over $\Z$. Demazure's theorem \cite[Expos{\'e} XXIII, Cor. 5.4]{SGA3.3} ensures that $\G_\Phi^P$ is unique up to isomorphism. A detailed proof of existence is also given in \cite[Expos{\'e} XXV]{SGA3.3}. (See also Lusztig's more recent approach to Kostant's construction \cite{Lusztig}.) We summarize the discussion with the following.

\begin{def/thm*}[\cite{Chevalley, Ree, SGA3.3, Kostant}]
Given a reduced root system $\Phi$ and a lattice $P$ with $P_{ad} \subseteq P \subseteq P_{sc}$, the \emph{Chevalley--Demazure group scheme} of type $(\Phi, P)$ is the split, semi-simple, affine group scheme $\G_\Phi^P$ defined over $\Z$ such that, for any field $k$, the split, semi-simple, linear algebraic group of type $\Phi$ and defined over $k$ is isomorphic to $\G_\Phi^P \otimes_{\Z} k$. 
\end{def/thm*}

A \emph{Chevalley--Demazure group} is the group of $R$-points $\G_\Phi^P(R)$ of some Chevalley--Demazure group scheme $\G_\Phi^P$ for some (commutative) ring $R$ (with unity). Of course, the two extreme cases of $P$ deserve special names. If $P = P_{ad}$, the root lattice, then $\G_\Phi^P =: \G_\Phi^{ad}$ is said to be of \emph{adjoint type}. If $P = P_{sc}$, the full lattice of weights of $\Lie{g}$, then $\G_\Phi^P$ is of \emph{simply connected type}. If, moreover, $\Phi$ is irreducible, then $\G_\Phi^{P_{sc}}$ is called \emph{universal}, and we write $\uCD := \G_\Phi^{P_{sc}}$. Since presentability problems for $\G_\Phi^P (R)$ are often reduced to the universal case, we shall be primarily concerned with $\G_\Phi^{sc}$.

The group scheme $\G_\Phi^P$ has the following properties. Let $y$ be an independent variable. For each $\alpha \in \Phi$ we get a monomorphism of the additive group scheme $\Addi = \Hom(\Z[y],-)$ into $\Hom_\Z(\Z[G_\C,P],-) = \G_\Phi^P$. Fix a ring $R$. Given an element $r \in (R,+) \cong \Addi(R)$, we denote its image under the map above by $x_\alpha(r) \in \G_\Phi^P(R)$. The \emph{unipotent root subgroup} associated to $\alpha$ is defined as $\mf{X}_\alpha := \gera{x_\alpha(r) \tq r \in R} \leq \G_\Phi^P(R)$, which is isomorphic to $\Addi(R) = (R,+)$. Furthermore, the map $\Addi \into \G_\Phi^P$ can be chosen so that
\[
\begin{pmatrix}
1 & r \\
0 & 1
\end{pmatrix} \in \SL_2(R) \mapsto x_\alpha(r) \mbox{ and }
\begin{pmatrix}
1 & 0 \\
r & 1
\end{pmatrix} \in \SL_2(R) \mapsto x_{-\alpha}(r),
\]
so we obtain an isomorphism from the subgroup either of $\SL_2(R)$ or of ${\rm PGL}_2(R)$ generated by elementary matrices to $\gera{\mf{X}_\alpha, \mf{X}_{-\alpha}} \leq \G_\Phi^P(R)$, the former being the case if $\G_\Phi^P = \uCD$. We then define the \emph{elementary subgroup} $E_\Phi^P$ of $\G_\Phi^P$ to be its subgroup generated by all unipotent root elements, that is $E_\Phi^P = \gera{\mf{X}_\alpha \tqalt \alpha \in \Phi} \leq \G_\Phi^P(R)$. In the Chevalley--Demazure setting, it is the analogous of the subgroup of elementary transvections of $\SL_n$. The groups $E_\Phi^P(R)$ and $\G_\Phi^P(R)$ need not coincide in general, but they are known to be equal in some important cases---perhaps most prominently in the case where $\G_\Phi^P$ is universal and $R$ is a field.

The maps from $\Addi$ into the $\mf{X}_\alpha \leq \G_\Phi^P$ also induce, for each $\alpha \in \Phi$, an embedding of the multiplicative group
$$\Mult \cong \begin{pmatrix} * & 0 \\ 0 & *^{-1} \end{pmatrix} \into \G_\Phi^P.$$
For $u \in R^\times$ we denote by $h_\alpha(u)$ the image of the matrix
\[
\begin{pmatrix} u & 0 \\ 0 & u^{-1} \end{pmatrix}
\in \SL_2(R)
\] 
under the map above. We call $\mc{H}_\alpha := \gera{\set{h_\alpha(u) \tq u \in R^\times}} \leq \G_\Phi^P(R)$ a \emph{semi-simple root subgroup}, and it is clearly a subtorus of $\G_\Phi^P(R)$. The key feature of a universal group is that $\mc{H} := \gera{\mc{H}_\alpha \tq \alpha \in \Phi}$ is a maximal split torus of $\uCD(R)$.

Two root subgroups $\mf{X}_\alpha, \mf{X}_\beta$ with $\alpha \neq -\beta$ are related by the \emph{Chevalley commutator formula}. If $x_\alpha(r) \in \mf{X}_\alpha$ and $x_\beta(s) \in \mf{X}_\beta$, then
\begin{equation} \label{ChevalleyRel}
[x_\alpha(r),x_\beta(s)] =
\begin{cases} \displaystyle\underset{m\alpha + n\beta \in \Phi}{\prod_{m, n > 0}} x_{m\alpha + n\beta}(r^m s^n)^{C_{m,n}^{\alpha,\beta}} & \mbox{if } \alpha + \beta \in \Phi; \\
1 & \mbox{otherwise,} \end{cases}
\end{equation}
where the powers $C_{m,n}^{\alpha,\beta}$, called \emph{structure constants}, always belong to $\set{0,\pm 1, \pm 2, \pm 3}$ and do not depend on $r$ nor on $s$, but rather on $\alpha, \beta$ and on the chosen total order on the set of simple roots $\Delta \subset \Phi$.

Steinberg derives in \cite[Chapter 3]{Steinberg} a series of consequences of the commutator formulae, known as Steinberg relations. Among these, we highlight the ones that relate the subtori $\mc{H}_\beta$ to the root subgroups $\mf{X}_\alpha$. Given $h_\beta(u) \in \mc{H}_\beta$ and $x_\alpha(r) \in \mf{X}_\alpha$, the following conjugation relation holds.
\begin{equation} \label{SteinbergRel}
h_\beta(u) x_\alpha(r) h_\beta(u)^{-1} = x_\alpha(u^{(\alpha, \beta)}r),
\end{equation}
where $(\alpha,\beta) \in \set{0, \pm 1, \pm 2, \pm 3}$ is the corresponding Cartan integer from Chevalley's Theorem.

Let $W$ be the Weyl group associated to $\Phi$. The conjugation relations above behave well with respect to both the action of $W$ on the roots and the conjugation action of (the ``canonical'' image of) $W$ on semi-simple and unipotent root elements. More precisely, let $\alpha$ be a root and $r_\alpha \in W$ be the associated reflection, and consider the element $w_\alpha \in E_\Phi^{P}(R)$ given by
\[
 w_\alpha := x_\alpha(1) x_{-\alpha}(1)^{-1} x_\alpha (1) = \mbox{ image of the matrix }
 \begin{pmatrix}
  0 & 1 \\
  -1 & 0
 \end{pmatrix}
 \in \SL_2(R) 
\]
under the map $E_2(R) \to \gera{\mf{X}_\alpha, \mf{X}_{-\alpha}}$ above. Then, given arbitrary roots $\beta, \gamma \in \Phi$, one has
\begin{align} \label{SteinbergReflection}
 \begin{split}
 h_{r_\alpha(\gamma)}(v) x_{r_\alpha(\beta)}(s) h_{r_\alpha(\gamma)}(v)^{-1} & = w_\alpha (h_\gamma(v) x_\beta(s)^{\pm 1} h_\gamma(v)^{-1}) w_\alpha^{-1} \\
 & = x_{r_\alpha(\beta)}(v^{(\beta, \gamma)}s)^{\pm 1},
 \end{split}
 \end{align}
where the sign $\pm 1$ above does not depend on $v \in R^\times$ nor on $s \in R$.

\subsection{Parabolic subgroups} \label{parabolicos} In the structure theory of Chevalley--Demazure and algebraic groups, their parabolic subgroups play an important role \cite{TitsBN, SGA3.3, MaTe}. They allow for different characterizations depending on the base ring \cite{SuzukiParabs, VavilovParabs, BorelLAG}. When defined over fields, a subgroup $\P \leq \uCD$ is called parabolic whenever the variety $\uCD / \P$ is complete. In the theory of buildings, standard parabolic subgroups arise as stabilizers of panels of a fixed fundamental chamber and are intimately related to parabolic subgroups of the corresponding Weyl group. Equivalent to those definitions over fields, $\P$ is said to be parabolic if it contains a Borel subgroup, which is a maximal, connected, soluble algebraic subgroup of $\uCD$. They always contain a maximal split torus. In this work, we consider those ``classical'' parabolics, most easily described via sets of simple roots.

\begin{dfn} \label{DefParab}
A \emph{standard} parabolic subgroup $\P(R) = \P_I(R)$ of a universal elementary Chevalley--Demazure group $\ueCD(R)$ is a group of the form
\[
\P_I(R) = \gera{\mc{H}, \mf{X}_\delta \tqalt \delta \in \Phi^+ \cup \Phi_I}
\]
for some subset of simple roots $I \subseteq \Delta$. The parabolic groups $\ueCD(R) = \P_\Delta(R)$ and $\mcB_\Phi(R) = \P_\vazio(R)$ are called \emph{trivial}, the latter also being known as a standard Borel subgroup of $\ueCD(R)$.
\end{dfn}

It is well-known that $\P_I(R)$ admits a \emph{Levi decomposition} \cite[Exp. XXVI, Prop. 1.6]{SGA3.3}, that is, it splits as a semi-direct product $\P_I(R) = \U_I(R) \rtimes \L_I(R)$ with
\begin{align*}
\U_I(R) = \gera{\mf{X}_\gamma \tqalt \gamma \in \Phi^+ \backslash \Phi_I},\\
\L_I(R) = \gera{\mc{H}, \mf{X}_\alpha \tqalt \alpha \in \Phi_I}.
\end{align*}
The normal subgroup $\U_I(R)$ is called the \emph{unipotent radical} of $\P_I(R)$---it is always nilpotent and admits a filtration via levels of roots with respect to the defining subset $I \subseteq \Delta$ (see, for instance, \cite{MaTe} for the case of algebraically closed fields or \cite[Exp. XXVI, Sec. 2]{SGA3.3} for the general case). The group $\L_I(R)$ is called the \emph{Levi factor} of $\P_I(R)$. When defined over a field, $\L_I$ is a reductive algebraic group and its derived subgroup will be simply-connected because $\uCD$ is universal. A good example to keep in mind is that of the parabolic subgroups in type $A_n$.

\begin{exm} \label{elementar}
Suppose $R$ is a field and $\rk(\Phi) = n - 1 \geq 2$. Then $E_{A_{n-1}}^{sc}(R) = \SL_n(R)$ with a set of simple roots given by $\Delta = \set{\alpha_1,\ldots,\alpha_{n-1}}$ for $\alpha_i = e_i - e_{i+1},~ 1 \leq i \leq n-1$, where $\set{e_j}_{j=1}^n \subseteq \R^n$ is the canonical basis.\\
If $i \neq j$, we denote by $\Eij{ij}(r)$ the matrix of $\SL_n(R)$ whose only non-zero entry is $r \in R$ in the position $ij$. The corresponding \emph{elementary matrix} is defined as $\eij{ij}(r) := \mbf{I}_n + \Eij{ij}(r)$. Commutators between elementary matrices have the following properties that are easily checked:
\[
[\eij{ij}(r),\eij{kl}(s)^{-1}] = [\eij{ij}(r),\eij{kl}(s)]^{-1} \mbox{ and}
\]
\[
[\eij{ij}(r),\eij{kl}(s)] =
\begin{cases}
\eij{il}(rs) & \mbox{if } k=l;\\
1 & \mbox{if } i \neq l, k \neq j.
\end{cases}
\]
Via the usual identification $\eij{i,i+1}(r) \longleftrightarrow x_{\alpha_i}(r)$ of elementary matrices with unipotent root elements and making use of the above, we iteratively recover all unipotent root subgroups as well as the commutator formulae \bref{ChevalleyRel} in type $A_{n-1}$. For instance, we can see that $x_{\alpha_i + \alpha_{i+1}}(r)$  $ = [\eij{i,{i+1}}(r),\eij{i+1,i+2}(1)]$, and the commutator formulae assume the simpler form 
\[
[x_\alpha(r),x_\beta(s)] =
\begin{cases} x_{\alpha + \beta}(r s) & \mbox{if } \alpha + \beta \in \Phi; \\
1 & \mbox{otherwise}. \end{cases}
\]
We thus obtain that the standard parabolic subgroups of $\SL_n(R)$ are of the form
 \[
\P_I(R) =
 \begin{pmatrix}
 L_1 & * & * & \cdots & * \\
 \mbf{0} & L_2 & * & \cdots & \vdots \\
 \vdots & ~ & \ddots & ~ & * \\
 \mbf{0} & \cdots & \mbf{0} & ~ & L_k
 \end{pmatrix}
 \leq \SL_n(R),
 \]
 i.e. $\P_I(R)$ is a subgroup of ``block upper triangular matrices''. The letters $L_i$ together represent possible matrix entries (not necessarily subgroups!) that will constitute the ``generalized block diagonal'' of the given parabolic subgroup, and the Levi factor $\L_I(R)$ is the subgroup generated by the diagonal matrices of $\SL_n(R)$ and the $L_i$ blocks. The condition $I \subsetneq \Delta$ would imply that the number of blocks, $k$, is at least $2$. If $I = \vazio$, all the blocks have size 1 and so $\P_\vazio(R)$ is the group of upper triangular matrices of $\SL_n(R)$. If $I \neq \vazio$, then at least one block $L_i$ is a square of size $\geq 2$, so it consists of invertible square matrices with determinant 1. In other words, it is isomorphic to some $\SL_{n_i}$ with $n_i < n$.\\
 For instance, suppose $n = 6$ and $I = \set{\alpha_1, \alpha_3}$. Then $\P_I(R) = \P_{\set{\alpha_1, \alpha_3}}(R) \leq \SL_6(R)$ is given by
 \[
 \P_{\set{\alpha_1,\alpha_3}}(R) = 
  \begin{pmatrix}
   * & * & * & * & * & * \\
   * & * & * & * & * & * \\
   0 & 0 & * & * & * & * \\
   0 & 0 & * & * & * & * \\
   0 & 0 & 0 & 0 & * & * \\
   0 & 0 & 0 & 0 & 0 & * 
  \end{pmatrix}.
 \]
 Its unipotent radical and Levi factor are given, respectively, by
 \[
  \U_{\set{\alpha_1, \alpha_3}}(R) =
    \begin{pmatrix}
   1 & 0 & * & * & * & * \\
   0 & 1 & * & * & * & * \\
   0 & 0 & 1 & 0 & * & * \\
   0 & 0 & 0 & 1 & * & * \\
   0 & 0 & 0 & 0 & 1 & * \\
   0 & 0 & 0 & 0 & 0 & 1 
  \end{pmatrix},~
  ~ \L_{\set{\alpha_1, \alpha_3}}(R) =
    \begin{pmatrix}
   * & * & \multicolumn{1}{|c}{0} & 0 & 0 & 0 \\
   * & * & \multicolumn{1}{|c}{0} & 0 & 0 & 0 \\ \cline{1-4}
   0 & 0 & \multicolumn{1}{|c}{*} & * & \multicolumn{1}{|c}{0} & 0 \\
   0 & 0 & \multicolumn{1}{|c}{*} & * & \multicolumn{1}{|c}{0} & 0 \\ \cline{3-5}
   0 & 0 & 0 & 0 & \multicolumn{1}{|c}{*} & \multicolumn{1}{|c}{0} \\ \cline{5-6}
   0 & 0 & 0 & 0 & 0 & \multicolumn{1}{|c}{*} 
  \end{pmatrix}.
 \]
We observe that the Levi factor $\L_{\set{\alpha_1, \alpha_3}}(R)$ is generated by the following subgroups of $\SL_6(R)$.
\begin{align*}
 L_1 =
    \begin{pmatrix}
   * & * & 0 & 0 & 0 & 0 \\
   * & * & 0 & 0 & 0 & 0 \\
   0 & 0 & 1 & 0 & 0 & 0 \\
   0 & 0 & 0 & 1 & 0 & 0 \\
   0 & 0 & 0 & 0 & 1 & 0 \\
   0 & 0 & 0 & 0 & 0 & 1 
  \end{pmatrix},~ &
  ~ H_2 =
   \begin{pmatrix}
   1 & 0 & 0 & 0 & 0 & 0 \\
   0 & \diamond & 0 & 0 & 0 & 0 \\
   0 & 0 & \diamond^{-1} & 0 & 0 & 0 \\
   0 & 0 & 0 & 1 & 0 & 0 \\
   0 & 0 & 0 & 0 & 1 & 0 \\
   0 & 0 & 0 & 0 & 0 & 1 
  \end{pmatrix}, \\
  L_3 =
    \begin{pmatrix}
   1 & 0 & 0 & 0 & 0 & 0 \\
   0 & 1 & 0 & 0 & 0 & 0 \\
   0 & 0 & * & * & 0 & 0 \\
   0 & 0 & * & * & 0 & 0 \\
   0 & 0 & 0 & 0 & 1 & 0 \\
   0 & 0 & 0 & 0 & 0 & 1 
  \end{pmatrix},~ &
  ~ H_{4,5} =
    \begin{pmatrix}
   1 & 0 & 0 & 0 & 0 & 0 \\
   0 & 1 & 0 & 0 & 0 & 0 \\
   0 & 0 & 1 & 0 & 0 & 0 \\
   0 & 0 & 0 & * & 0 & 0 \\
   0 & 0 & 0 & 0 & * & 0 \\
   0 & 0 & 0 & 0 & 0 & * 
  \end{pmatrix}.
\end{align*}
In this case (recall that $R$ is a field!), we have the following identifications using the notation from Section \ref{jargao}.
\begin{align*}
 L_1 = \gera{\mf{X}_{\alpha_1}, \mf{X}_{-\alpha_1}} \cong \SL_2(R),~ & ~ H_2 = \H_{\alpha_2} \cong \Mult(R) \\
 L_3 = \gera{\mf{X}_{\alpha_3}, \mf{X}_{-\alpha_3}} \cong \SL_2(R),~ & ~ H_{4,5} = \gera{\H_{\alpha_4}, \H_{\alpha_5}} \cong \Mult(R)^2
\end{align*}
and
\begin{align*}
\U_{\set{\alpha_1,\alpha_3}}(R) = \langle & \mf{X}_{\alpha_2},~ \mf{X}_{\alpha_4},~ \mf{X}_{\alpha_5},~ \mf{X}_{\alpha_1 + \alpha_2},~ \mf{X}_{\alpha_2 + \alpha_3},~ \mf{X}_{\alpha_3 + \alpha_4},~ \mf{X}_{\alpha_4 + \alpha_5}, \\
 & \mf{X}_{\alpha_1 + \alpha_2 + \alpha_3},~ \mf{X}_{\alpha_2 + \alpha_3 + \alpha_4},~ \mf{X}_{\alpha_3 + \alpha_4 + \alpha_5}, \\
 & \mf{X}_{\alpha_1 + \alpha_2 + \alpha_3 + \alpha_4},~ \mf{X}_{\alpha_2 + \alpha_3 + \alpha_4 + \alpha_5}, \\ 
 & \mf{X}_{\alpha_1 + \alpha_2 + \alpha_3 + \alpha_4 + \alpha_5}\rangle. 
\end{align*}
\end{exm}

$~$

The decomposition for the Levi factor $\L_{\set{\alpha_1, \alpha_3}}(R)$ from the example above holds, of course, in a more general context. Recall that the root system $\Phi$ has an arbitrary, but fixed, choice of (totally ordered) simple roots $\Delta$. Given a subset $X \subseteq \Phi$, we let $\Phi_X \subseteq \Phi$ denote the subsystem generated by $X$. Viewing $\Delta$ as the set of vertices of its Dynkin diagram $\mc{D}_\Delta$, if $I$ is a subset of simple roots, we write $\Adj(I)$ for the set of simple roots \emph{not in} $I$ that are adjacent to some (not necessarily the same) element of $I$. In symbols,
\begin{align*}
 \Adj(I) := \{ \delta \in \Delta \backslash I \tq & \exists \alpha \in I \mbox{ for which there is an edge in } \mc{D}_\Delta \\
& \mbox{connecting } \delta \mbox{ and } \alpha \}.
\end{align*}
Now, let $\vazio \neq I \subsetneq \Delta$. This subset $I$ of simple roots generates a subdiagram $\mc{I}$ in the Dynkin diagram $\mc{D}_\Delta$. Denote by $I_1, \ldots, I_k$ the (pairwise disjoint) subsets of $I$ that span the connected components of $\mc{I}$ in $\mc{D}_\Delta$. We observe that
\[
 \Phi_I = \Phi_{I_1} \cupdot \Phi_{I_2} \cupdot \cdots \cupdot \Phi_{I_k}.
\]
 It then follows from Chevalley's formula \bref{ChevalleyRel} and Steinberg's relations \bref{SteinbergRel} that the Levi factor $\L_I(R) \leq \P_I(R) \leq \ueCD(R)$ is in fact an extension of a direct product of elementary Chevalley--Demazure groups, of rank smaller than $\rk(\Phi)$, by a torus. Namely,
\[
 \L_I(R) = E_{\Phi_I}^{sc}(R) \rtimes \gera{\H_{\alpha} \tqalt \alpha \in \Delta \backslash I} = \left( \prod_{j = 1}^k E_{\Phi_{I_j}}^{sc}(R) \right) \rtimes \gera{\H_{\alpha} \tqalt \alpha \in \Delta \backslash I}.
\]
If $ I = \vazio$, then $\L_I(R)$ coincides with the standard torus $\H = \gera{\H_\alpha \tqalt \alpha \in \Delta}$.

\subsection{Terminology} \label{notacao} For convenience, we collect below the notation to be used throughout the remainder of this paper.

If $X$ is a subset of a group $G$, we denote by $\gera{X}$ the subgroup of $G$ generated by $X$. Similarly, given a set $X$ and a set $R$ of (reduced) words on the alphabet over $X$, we denote by $\gera{X \tq R}$ the group with generating set $X$ and defining relators $R$. That is, $\gera{X \tq R} \cong F_X / \langle \langle R \rangle \rangle$, where $F_X$ is the free group with basis $X$ and $\langle \langle R \rangle \rangle$ is the normal closure of $R \subseteq F_X$. Standard results on group presentations, in particular von Dyck's theorems, will be used freely without further references (see e.g. \cite{Cohen, Johnson}).

\begin{itemize}
 \item $R$ always denotes a commutative ring with $1$;
 \item $\Addi$ and $\Mult$ denote, respectively, the additive and the multiplicative affine group schemes over $\Z$ (see e.g. \cite{Waterhouse}). So, given a ring $R$, we have $\Addi(R) = (R,+)$ and $\Mult(R) = (R^\times, \cdot)$;
 \item The inner product of a Euclidean space is denoted by $\gera{\cdot, \cdot}$;
 \item $\Phi$ is a reduced, irreducible root system with an arbitrary, but fixed, choice of subset $\Delta \subseteq \Phi$ of simple roots with a total order compatible with heights of roots (see e.g. \cite[13.11]{MaTe}) that extends to the whole $\Phi$;
 \item Using the order above, given a subset $X\subseteq \Phi$ we write $X^+$ (respectively, $X^{-}$) for the subset of positive (respectively, negative) roots of $X$;
 \item If $X \subseteq \Phi$, then $\Phi_X \subseteq \Phi$ denotes the root subsystem of $\Phi$ generated by $X$, i.e. $\Phi_X = \vspan_\Z(X) \cap \Phi \subseteq \R^{\rk(\Phi)}$;
 \item Given two roots $\alpha, \beta \in \Phi \subset \R^{\rk(\Phi)}$, the Cartan integer $(\alpha, \beta)$ is given by $(\alpha, \beta) = 2\frac{\gera{\alpha,\beta}}{\gera{\beta,\beta}}$;
 \item $\uCD$ denotes the universal Chevalley--Demazure group scheme over $\Z$ with elementary subgroup $\ueCD$;
 \item For $\alpha \in \Phi$, the unipotent root subgroup $\mf{X}_\alpha \leq \ueCD(R)$ is the subgroup of $\uCD(R)$ generated by the unipotent root elements $x_\alpha(r), r \in R$, so $\ueCD(R) = \gera{\mf{X}_\alpha \tqalt \alpha \in \Phi}$;
 \item The semi-simple root subgroup $\mc{H}_\alpha \leq \ueCD(R)$ is the subgroup of $\uCD(R)$ generated by the semi-simple root elements $h_\alpha(u), u \in R^\times$. They define the standard torus $\mc{H} = \gera{\mc{H}_\alpha \tq \alpha \in \Phi}$ in $\ueCD(R)$, which has maximal rank;
 \item $I$ always denotes denotes a subset of the simple roots $\Delta$;
 \item Viewing $I \subseteq \Delta$ as a vertex subset of the Dynkin diagram of $\Phi$, we let $\Adj(I) \subseteq \Delta \backslash I$ be the set of simple roots $\alpha \notin I$ which are \emph{adjacent} to some (not necessarily the same) element of $I$. We define further the \emph{non-adjacent roots} as the complement $\nonAdj(I) = \Delta \backslash (I \cup \Adj(I))$ and the \emph{extension of} $I$ to be $\Ext(I) := \Delta \backslash \Adj(I) = I \cup \nonAdj(I)$;
 \item $\P_I \leq \ueCD$ is the corresponding standard parabolic subgroup of $\ueCD$ (Def. \ref{DefParab}). It always contains the torus $\mc{H} \leq \ueCD$ of rank $\rk(\Phi)$;
 \item $\P_I = \U_I \rtimes \L_I$ is the Levi decomposition of $\P_I$, i.e. $\U_I$ is its unipotent radical and $\L_I$ denotes its Levi subgroup;
 \item The minimal parabolic $\P_\vazio(R)$ is also denoted by $\B_\Phi(R)$, called the standard Borel subgroup of $\ueCD(R)$. The standard Borel subgroup in type $A_1$ is denoted by $\Bzero := \B_{A_1}(R) = \P_\vazio(R) \leq \SL_2(R)$.
\end{itemize}

\section{Retracts and the extended Levi factor} \label{Levi}

\noindent

An object $Y$ in a category $\mc{C}$ is called a retract of $X \in \Obj(\mc{C})$ if there exists a morphism $r \in \Hom_{\mc{C}}(X,Y)$ which admits a morphism $\iota \in \Hom_{\mc{C}}(Y,X)$ as a section. For pointed, path-connected, topological spaces, the existence of a retraction $r : X \to Y$ implies, intuitively, that the fundamental group of $Y$ cannot be ``worse'' than the fundamental group of $X$. For categories of algebraic objects such as $R$-modules or groups, retractions correspond precisely to split short exact sequences. So for groups, a retraction $r : G \to Q$ is just another name for the semi-direct product $G \cong N \rtimes Q$, where $N = \ker(r)$ and $Q$ acts on $N$ via conjugation.

H. {\AA}berg observed in the 80's that a retract $Q$ of a group $G$ inherits certain homological properties of $G$. For generators and relations, a similar inheritance has been long known: the image of a finitely generated group is obviously finitely generated, and if $G$ is finitely presented, then so is $Q$ by the following result due to J. Stallings.

\begin{lem}[{\cite[Lemma 1.3]{Wall}}] \label{Stallings}
 If $Q$ is a retract of a finitely presented group, then $Q$ is finitely presented.
\end{lem}

So the existence of ``nice'' (e.g. finite or compact) generating sets or presentations of $Q$ is a necessary condition to obtain ``nice'' generating sets or presentations of $G$. In particular, Lemma \ref{Stallings} implies that the Levi factor $\L_I(R)$ is always finitely presented whenever the whole parabolic $\P_I(R)$ is so, and we obtain from the description of $\L_I(R)$ given at the end of Section \ref{parabolicos} that, in our set-up, the torus $\H$ (and whence $\Mult(R) = R^\times$) is finitely generated. 

Going the other direction, however, is in general a delicate problem. For the parabolic $\P_I(R) = \U_I(R) \rtimes \L_I(R)$, the example in Section \ref{exemplao} shows that $\L_I(R)$ might fail to detect qualitative properties of $\P_I(R)$ such as finite or compact presentability. To remedy the situation, we instead consider a slightly larger retract of $\P_I(R)$.

Recall that $\Adj(I) \subseteq \Delta \backslash I$ is the set of simple roots $\alpha \notin I$ adjacent to some (not necessarily the same) element of $I$. The \emph{non-adjacent} roots are the complement 
\[
\nonAdj(I) = \Delta \backslash (I \cup \Adj(I)) = \set{\alpha \in \Delta \backslash I \tq \alpha \mbox{ is adjacent to no element of } I}, 
\]
and the \emph{extension} of the given set $I$ is defined as $\Ext(I) = \Delta \backslash \Adj(I) = I \cup \nonAdj(I)$.

\begin{dfn} \label{defLE}
An \emph{extended Levi factor}, denoted $\LE_n(R)$ or $\LE_I(R)$, is a subgroup of a standard parabolic $\P_I(R) \leq \ueCD(R)$ which is generated by the standard torus $\H$ and either by a single root subgroup $\mf{X}_\alpha$ with $\alpha \in \Delta$, in case $I$ is empty, or by the root subgroups $\mf{X}_{\alpha}$ for $\alpha \in \Phi_I$ together with the non-adjacent positive root subgroups $\mf{X}_{\beta}$ with $\beta \in \Phi_{\nonAdj(I)}^+$, if $I \neq \vazio$. In symbols,
\[
\begin{cases}
 \LE_n(R) := \gera{\H,~ \mf{X}_{\alpha_n} \tqalt \alpha_n \mbox{ is the } n\mbox{-th root of } \Delta}, & \mbox{ if } I = \vazio;\\
 \LE_I(R) := \gera{\H,~ \mf{X}_\alpha,~ \mf{X}_\beta \tqalt \alpha \in \Phi_I,~ \beta \in \Phi_{\nonAdj(I)}^+} & \mbox{otherwise}.
\end{cases}
\]
\end{dfn}

We reserve the notation $\LE_I(R)$ for the case $I \neq \vazio$. Of course, by the very definition there is a unique extended Levi factor when $I \neq \vazio$ and there exist $\rk(\Phi)$ extended Levi factors if $I = \vazio$.  Since $\L_I(R) = \gera{\H,~ \mf{X}_\alpha \tqalt \alpha \in \Phi_I}$ and $\L_\vazio(R) = \H$, one has $\LE_I(R) \supseteq \L_I(R)$ and $\LE_n(R) \supsetneq \L_\vazio(R)$, and we see that the definition given above coincides with the one given in the introduction. Furthermore, we get the following \emph{split} short exact sequences.
\[
 \mf{X}_{\alpha_n} \into \LE_n(R) \onto \L_\vazio(R) = \H \cong \LE_n(R)/ \mf{X}_{\alpha_n}
\]
and
\[
 \gera{\mf{X}_\beta \tqalt \beta \in \Phi_{\nonAdj(I)}^+} \into \LE_I(R) \onto \L_I(R) \cong \LE_I(R) / \gera{\mf{X}_\beta \tqalt \beta \in \Phi_{\nonAdj(I)}^+}.
\]
By Lemma \ref{Stallings}, if an extended Levi factor is finitely presented, then so is the Levi factor itself.

In the language of Sections \ref{exemplao} and \ref{main}, for the case $I \neq \vazio$, the root subgroups $\mf{X}_\beta$ with $\beta \in \nonAdj(I)$ are the generators of the \textcolor{red}{triangular blocks} of $\P_I(R)$, whereas $\L_I(R)$ is generated by both the \textcolor{blue}{regular blocks}---generated by the $\mf{X}_\alpha$ with $\alpha \in \Phi_I$---and the torus $\H$.

\begin{exm}
We describe the extended Levi factor for the parabolic subgroup
\[
 \P_{\set{\alpha_1,\alpha_3}}(R) = 
  \begin{pmatrix}
   * & * & * & * & * & * \\
   * & * & * & * & * & * \\
   0 & 0 & * & * & * & * \\
   0 & 0 & * & * & * & * \\
   0 & 0 & 0 & 0 & * & * \\
   0 & 0 & 0 & 0 & 0 & * 
  \end{pmatrix}
  \leq \SL_6(R)
\]
 from Example \ref{elementar} ($R$ a field). As observed before, we have two subgroups which are generated by root subgroups with roots from $\Phi_{\set{\alpha_1,\alpha_3}}$, namely $L_1 = \gera{\mf{X}_{\alpha_1},~ \mf{X}_{-\alpha_1}}$ and $L_3 = \gera{\mf{X}_{\alpha_3},~ \mf{X}_{-\alpha_3}}$. Those are precisely the \textcolor{blue}{regular blocks} from $\P_{\set{\alpha_1, \alpha_3}}(R)$. Now, the only simple root which is not adjacent to $I = \set{\alpha_1, \alpha_3}$ is the last one, $\alpha_5$. So $\nonAdj(I)$ is the singleton $\set{\alpha_5}$ and one has $\Phi_{\nonAdj(I)} = \set{\pm \alpha_5}$ and $\Phi_{\nonAdj(I)}^+ = \set{\alpha_5}$, whence the \textcolor{red}{triangular block} of $\P_{\set{\alpha_1, \alpha_3}}(R)$ is just the root subgroup $\mf{X}_{\alpha_5}$. Pictorially,
\[
 \textcolor{blue}{L_1} =
    \begin{pmatrix}
   \textcolor{blue}{*} & \textcolor{blue}{*} & 0 & 0 & 0 & 0 \\
   \textcolor{blue}{*} & \textcolor{blue}{*} & 0 & 0 & 0 & 0 \\
   0 & 0 & 1 & 0 & 0 & 0 \\
   0 & 0 & 0 & 1 & 0 & 0 \\
   0 & 0 & 0 & 0 & 1 & 0 \\
   0 & 0 & 0 & 0 & 0 & 1 
  \end{pmatrix},~
  ~\textcolor{blue}{L_3} =
    \begin{pmatrix}
   1 & 0 & 0 & 0 & 0 & 0 \\
   0 & 1 & 0 & 0 & 0 & 0 \\
   0 & 0 & \textcolor{blue}{*} & \textcolor{blue}{*} & 0 & 0 \\
   0 & 0 & \textcolor{blue}{*} & \textcolor{blue}{*} & 0 & 0 \\
   0 & 0 & 0 & 0 & 1 & 0 \\
   0 & 0 & 0 & 0 & 0 & 1 
  \end{pmatrix} \mbox{ and}
  \]
  \[
   \textcolor{red}{\mf{X}_{\alpha_5}} =
    \begin{pmatrix}
   1 & 0 & 0 & 0 & 0 & 0 \\
   0 & 1 & 0 & 0 & 0 & 0 \\
   0 & 0 & 1 & 0 & 0 & 0 \\
   0 & 0 & 0 & 1 & 0 & 0 \\
   0 & 0 & 0 & 0 & 1 & \textcolor{red}{*} \\
   0 & 0 & 0 & 0 & 0 & 1 
  \end{pmatrix},
  \mbox{ so } \LE_{\set{\alpha_1, \alpha_3}}(R) =
    \begin{pmatrix}
   \textcolor{blue}{*} & \textcolor{blue}{*} & \multicolumn{1}{|c}{0} & 0 & 0 & 0 \\
   \textcolor{blue}{*} & \textcolor{blue}{*} & \multicolumn{1}{|c}{0} & 0 & 0 & 0 \\ \cline{1-4}
   0 & 0 & \multicolumn{1}{|c}{\textcolor{blue}{*}} & \textcolor{blue}{*} & \multicolumn{1}{|c}{0} & 0 \\
   0 & 0 & \multicolumn{1}{|c}{\textcolor{blue}{*}} & \textcolor{blue}{*} & \multicolumn{1}{|c}{0} & 0 \\ \cline{3-6}
   0 & 0 & 0 & 0 & \multicolumn{1}{|c}{*} & \textcolor{red}{*} \\
   0 & 0 & 0 & 0 & \multicolumn{1}{|c}{\textcolor{red}{0}} & * 
  \end{pmatrix}
  \]
  because, besides containing the \textcolor{blue}{regular} and \textcolor{red}{triangular} blocks, an extended Levi factor also contains the torus---which in this case corresponds to the subgroup of diagonal matrices of $\SL_6(R)$.
\end{exm}

$~$

Of course, we still have to show that $\LE_I(R)$ fits our framework of retracts with respect to $\P_I(R)$. 

\begin{pps} \label{LEretrato}
 Suppose $I \neq \vazio$. There is a retract $r: \P_I(R) \onto \LE_I(R)$ with kernel $\mc{K}_I(R) = \gera{\mf{X}_\gamma \tqalt \gamma \in \Phi^+ \backslash \Phi_{\Ext(I)}}$. If $I = \vazio$, then one has a retract $r : \P_\vazio(R) = \B_\Phi(R) \onto \LE_n(R)$ with kernel $\mc{K}_n(R) = \gera{\mf{X}_\gamma \tqalt \gamma \in \Phi^+ \backslash \set{\alpha_n}}$.
\end{pps}

\begin{proof}
 Assume first that $I = \vazio$. By Chevalley's commutator formula \bref{ChevalleyRel} and Steinberg's relations \bref{SteinbergRel}, we see that $\mc{K}_n(R) \nsgp \B_\Phi(R)$ and $\LE_n(R) = \mf{X}_{\alpha_n} \rtimes \H \cong \B_\Phi(R) / \mc{K}_n(R)$. Suppose $ I \neq \vazio$. Again from \bref{ChevalleyRel} and \bref{SteinbergRel} it follows that $\K_I(R) \nsgp \P_I(R)$. Since $\Ext(I)$ is the disoint union of $I$ and $\nonAdj(I)$, we have that the unipotent root subgroups of $\LE_I(R)$ do not involve any $\mf{X}_\gamma \leq \mc{K}_I(R)$ and are furthermore partitioned into blocks which pairwise have only non-adjacent roots between them. Hence, $\LE_I(R)$ is a complement of $\mc{K}_I(R)$ in $\P_I(R)$ with trivial intersection, so the natural projection $\P_I(R) \onto \P_I(R)/\mc{K}_I(R)$ yields the desired retraction.
\end{proof}

It follows from Lemma \ref{Stallings} and Proposition \ref{LEretrato} that the finite presentability of $\LE_I(R)$ (or of all the $\LE_n(R)$) is a necessary condition for the finite presentability of the whole parabolic $\P_I(R)$. Proposition \ref{LEretrato} also implies that we can make use of the usual presentation of a semi-direct product to build a presentation for $\P_I$ out of presentations of $\LE_I$ and $\mc{K}_I$. For the Borel subgroup $\B_\Phi(R)$, i.e. the case $I = \vazio$, we will instead make use of a presentation for its unipotent radical $\U_\vazio(R) = \gera{\mf{X}_\gamma \tqalt \gamma \in \Phi^+}$. We shall therefore need convenient presentations for $\U_\vazio(R)$ and $\mc{K}_I(R)$ for the proof of Theorem \ref{A}, which we describe in the sequel.

Recall that the unipotent radical $\U_I(R)$---for arbitrary $I$---admits a well-known presentation obtained as follows. The unipotent root elements $x_\gamma(r)$, with $\gamma$ running over $\Phi^+ \backslash \Phi_I$ and $r \in R$, form the generating set; the defining relators are given by the Chevalley commutator formulae \bref{ChevalleyRel} and the additive condition $x_\gamma(r) \cdot x_\gamma(s) = x_\gamma(r+s)$ for all $\gamma \in \Phi^+ \backslash \Phi_I,~ r,s \in R$. For a geometric proof of this in the context of RGD systems one might proceed as in \cite[Chapter 8.10]{AbraBrown}. A more straightforward proof can be obtained by mimicking the arguments in \cite[Thm 2.a, see also closing remark 1]{AzadBarrySeitz} to describe a central series for $\U_I(R)$ as a nilpotent group (see also \cite[Prop. 17.3]{MaTe}).

Such presentation usually involves many more generators then needed (in fact, copies of the whole base ring $R$) and many relations. The additive structure of $R$ is encoded in the unipotent root subgroups, and its multiplicative structure shows up in the commutators. Picking a generating set for $\Addi(R) = (R, +)$, one can rewrite such a presentation by replacing the additive conditions by (copies of) defining relators of $\Addi(R)$ and then keeping only commutator formulae that involve the chosen generators of $\Addi(R)$. The drawback is that the commutators might become quite messy. These methods are explicitly illustrated, for instance, in \cite{BissDasgupta} for the unipotent radical of Borel subgroups in type $A_n$, though in that work in a very economical way. Since we are considering qualitative rather than quantitative properties, we make no attempt to decrease the numbers of generators and relators used.

\begin{lem} \label{presunip}
Let $T \subseteq R$ be a generating set for the underlying additive group $\Addi(R)$ of the ring $R$. For every pair $t^i,s^j$ with $t,s\in T$ and $i,j\in\Z$, fix an additive expression $m(t^i,s^j)$ in terms of $T$ for the product $t^i s^j$, that is, $m(t^i,s^j)$ is of the form
\[
m(t^i,s^j)= \sum a_u u(t^i,s^j), \mbox{ with } u = u(t^i,s^j) \in T \mbox{ and } a_u \in \Z,
\]
where all but finitely many $a_u$'s are $0$, and (the image of) $m(t^i,s^j)$ in $R$ equals $t^i s^j$.

The unipotent radical $\mc{U}_I(R)$ admits a presentation with generating set
\[
\mc{Y} = \set{x_\gamma(t) \tq t\in T, \gamma \in \Phi^+ \backslash \Phi_{I}}
\]
and a set of defining relators $\mc{S}$ given as follows. For all $\gamma, \eta \in \Phi^+ \backslash \Phi_I$,
\begin{equation} \label{U_S.1}
[x_\gamma(t), x_\eta(s)] = \begin{cases} \displaystyle\prod_{m\gamma + n\eta \in \Phi^+} \left( \prod_u^{} x_{m\gamma+n\eta}(u(t^m,s^n))^{a_u\cdot C^{\gamma,\eta}_{m,n}} \right), \mbox{ if } \gamma + \eta \in \Phi;\\ 
1 \mbox{ otherwise}, \end{cases}
\end{equation}
where $u = u(t^m, s^n)$ is as above, and
\begin{equation} \label{U_S.2}
\prod_{i=1}^n x_\gamma(t_{\lambda_i})^{a_i} = 1 \mbox{ whenever } \sum_{i=1}^n a_i t_{\lambda_i} = 0 \mbox{ in } R, \mbox{ where } a_i\in\Z,~t_{\lambda_i}\in T.
\end{equation}
\end{lem}

This gives us, in particular, a presentation for the unipotent radical $\U_\vazio(R)$ of the Borel subgroup $\B_\Phi(R) \leq \ueCD(R)$. Our next remark is that $\mc{K}_I(R)$ admits a presentation very similar to that of $\U_I(R)$ described above. The proof is analogous to that of the previous lemma and we include it here for the sake of completeness. The reader unfamiliar with the proof of Lemma \ref{presunip} might just adapt the proof of the lemma below to the set-up of \ref{presunip}.

\begin{lem} \label{presK}
Let $T \subseteq R$ and let the expressions $m(t^i,s^j)$ be as in Lemma \ref{presunip} and suppose $I \neq \vazio$. Then the kernel $\mc{K}_I(R)$ of the retraction $r : \P_I(R) \onto \LE_I(R)$ of Proposition \ref{LEretrato} admits a presentation with generating set
$$\mc{Y} = \set{x_\gamma(t) \tq t\in T, \gamma \in \Phi^+ \backslash \Phi_{\Ext(I)}}$$
and a set of defining relators $\mc{S}$ given by the same formulae \bref{U_S.1} and \bref{U_S.2} of Lemma \ref{presunip}, but now for all $\gamma, \eta \in \Phi^+ \backslash \Phi_{\Ext(I)}$.
\end{lem}

\begin{proof}
Given a positive root $\alpha \in \Phi^+$, we can write (uniquely)
\[
 \alpha = \sum_{\delta \in \Ext(I)} p_\delta \delta + \sum_{\gamma \in \Adj(I)} q_\gamma \gamma, \mbox{ with } p_\delta, q_\gamma \in \Nzero.
\]
We define the \emph{adjacency level of} $\alpha$, denoted $\rm{alvl}(\alpha)$, to be the integer $\mathrm{alvl}(\alpha) = \sum_{\gamma \in \Adj(I)} q_\gamma$ from the equation above. Let
\[
 \mc{K}_j = \gera{\mf{X}_\gamma \tqalt \gamma \in \Phi^+ \backslash \Phi_{\Ext(I)} \mbox{ has } \mathrm{alvl}(\gamma) \geq j}.
\]
By the commutator formuale \bref{ChevalleyRel}, we see that each $\mc{K}_j$ is normal in $\mc{K}_I(R)$ and each factor $\mc{K}_j / \mc{K}_{j+1}$ is canonically isomorphic to
\[
 \underset{\mathrm{alvl}(\gamma) = j}{\prod_{\gamma \in \Phi^+ \backslash \Phi_{\Ext(I)}}} \mf{X}_\gamma \cong \underset{\mathrm{alvl}(\gamma) = j}{\prod_{\gamma \in \Phi^+ \backslash \Phi_{\Ext(I)}}} \Addi(R).
\]
Hence, the subgroups $\mc{K}_j$ give a terminating central series for $\mc{K}_I(R)$ (though it may not coincide with the lower central series).

Now let $\til{\mc{K}}_I(R)$ be the group defined by the presentation stated in the lemma, with the decoration $ ~ \til{} ~ $ above the elements of the generating set (e.g. $\til{x}_\gamma(t)$ instead of $x_\gamma(t)$). Define analogously
\[
 \til{\mc{K}}_j = \gera{\set{\til{x}_\gamma(t) \tq t \in T \mbox{ and } \gamma \in \Phi^+ \backslash \Phi_{\Ext(I)} \mbox{ of } \mathrm{alvl}(\gamma) \geq j}}.
\]
By von Dyck's theorem, the obvious map
\[
f : \mc{Y} \longrightarrow \mc{K}_I(R)
\]
\[
\til{x}_\gamma(t) \mapsto x_\gamma(t)
\]
induces a surjection $\til{\mc{K}}_I(R) \onto \mc{K}_I(R)$---which we also call $f$ by abuse of notation---because the defining relations \bref{U_S.1} and \bref{U_S.2} hold in $\mc{K}_I(R)$ and each $\mf{X}_\gamma \cong \Addi(R)$ is generated by the $\set{x_\gamma(t) \tq t \in T}$. By \bref{U_S.1} we also see that $\til{\mc{K}}_j \nsgp \til{\mc{K}}_I(R)$ for every $j$, and $f$ restricts to surjections $\til{\mc{K}}_j \onto \mc{K}_j$. Furthermore, \bref{U_S.1} and \bref{U_S.2} imply that each factor $\til{K}_j / \til{K}_{j+1}$ is isomorphic to
\[
 \underset{\mathrm{alvl}(\gamma) = j}{\prod_{\gamma \in \Phi^+ \backslash \Phi_{\Ext(I)}}} \til{\mf{X}}_\gamma \cong \underset{\mathrm{alvl}(\gamma) = j}{\prod_{\gamma \in \Phi^+ \backslash \Phi_{\Ext(I)}}} \Addi(R),
\]
where the $\til{\mf{X}_\gamma}$ are defined in the obvious way. Such maps are all induced by $f$. So $\til{\mc{K}}_I(R)$ and $\mc{K}_I(R)$ are nilpotent groups with isomorphic (terminating) central series via isomorphisms induced by the same surjection. It then follows by induction on the nilpotency class that $f$ is an isomorphism.
\end{proof}

\section{Proof of the main result} \label{teoremao}

\noindent

The standing assumption of Theorem \ref{A}---and thus of this whole section---is that the standard parabolic subgroups of the universal elementary Chevalley--Demazure $\ueCD(R)$ \emph{are finitely generated}. We fix once and for all an arbitrary total order on the set of simple roots $\Delta \subset \Phi$ which extends to a total order on $\Phi$.

As seen in Section \ref{Levi}, finite presentability of $\P_I(R)$ implies that of $\LE_I(R)$ (or of all the $\LE_n(R)$ in case $I = \vazio$) by Lemma \ref{Stallings} and Proposition \ref{LEretrato}. It therefore remains to prove the converse to Theorem \ref{A}, stated below. 

\begin{thm} \label{avolta}
Let $\P_I(R) \leq \ueCD(R)$ be a standard parabolic subgroup and let $R$ be \QGff. Assume $I \neq \set{\alpha},~ \alpha$ long, in case $\Phi = G_2$. If $I \neq \vazio$, suppose the (unique) extended Levi factor $\LE_I(R)$ is finitely presented. In case $I = \vazio$, we assume each extended Levi factor $\LE_n(R)$ to admit a finite presentation. Then $\P_I(R)$ is finitely presented.
\end{thm}

The strategy to prove Theorem \ref{avolta} is as follows. We first discuss the structure of the base ring $R$ under our standing assumptions. Now, recall that a semi-direct product $G = N \rtimes_\phee Q$, with $N = \gera{ Y \tq S},~ Q = \gera{ X \tq R }$ and $\phee$ the homomorphism $Q \xrightarrow{\phee} \Aut(N)$ determining the action, admits the following presentation.
\begin{equation} \label{semidireto}
G = \gera{X \cup Y \tq S \cup R \cup \set{ x y x^{-1} \phee(y) \tq x \in X,~ y \in Y}}. \tag{*}
\end{equation}
We then consider the two cases, $I = \vazio$ and $I \neq \vazio$. In the former, $\P_I(R) = \P_\vazio(R) = \B_\Phi(R) = \U_\vazio(R) \rtimes \H$. Since $\H$ is a finitely generated abelian group, we can take a \emph{finite} presentation for $\H$ with finitely many semi-simple root elements as generators. Combining this with the presentation for $\U_\vazio(R)$ from Lemma \ref{presunip} and Steinberg's relations \bref{SteinbergRel}, we get a ``canonical'' (in general infinite) presentation for $\B_\Phi(R)$ as in \bref{semidireto}. On the other hand, using the fact that the extended Levi factors $\LE_n(R)$ are finitely presented, we obtain finite presentations for each subgroup $\mf{X}_\gamma \rtimes \H \leq \B_\Phi(R),~ \gamma \in \Phi^+$. Starting from those presentations, we appropriately add unipotent root elements and Chevalley and Steinberg relations to construct a finitely presented group $\til{\B}_\Phi(R)$. Finally, we apply von Dyck's theorem twice to show that $\til{\B}_\Phi(R)$ is isomorphic to $\B_\Phi(R)$.

For $I \neq \vazio$, we start by taking a finite presentation for the extended Levi factor $\LE_I(R)$, with generating set given by appropriately chosen unipotent and semi-simple root elements. The ``canonical'' presentation for $\P_I(R)$ is now given by the chosen presentation for $\LE_I(R)$ together with the presentation of $\mc{K}_I(R)$ from Lemma \ref{presK} and, of course, Chevalley and Steinberg relations. Then, we break down the proof in two further cases given by the {\QG} condition: If $\Bzero$ is finitely presented, we get finite presentations for each $\mf{X}_\gamma \rtimes \H,~ \gamma \in \Phi^+ \backslash \Phi_{\Ext(I)}$, and then proceed similarly to the previous case of $\B_\Phi(R)$; If $R$ is \NVBff, we construct a finitely presented group $\til{\P}_I(R)$ from $\LE_I(R)$ adding just the obvious generators from $\mc{K}_I(R)$ as a normal subgroup of $\P_I(R)$ and the necessary Chevalley and Steinberg relations, then proceed to show via commutator computations that $\til{\P}_I(R) \cong \P_I(R)$.

We aim to state the intermediate results distinguishing as little as possible the different types of root systems, which means we shall often write commutator formulae in their full generality. We therefore warn the reader to be armed with patience to face the lengthy notation battle ahead.

Let us begin with the following elementary results on root systems and root subgroups that will be needed in the sequel. 

\begin{lem}[{\cite[10.2 A]{Humphreys}}] \label{Humphreys}
 Given a positive, non-simple root $\gamma \in \Phi^+\backslash \Delta$, there exist a simple root $\alpha \in \Delta$ and a positive root $\beta \in \Phi^+$ such that $\alpha + \beta = \gamma$.
\end{lem}

\begin{lem} \label{adjacencia}
 If $\vazio \neq I \subsetneq \Delta$ and $\alpha \in \Adj(I)$, then there exist $\til{\alpha} \in \Phi_I$ and $\til{\beta} \in \Phi^+$ such that $\til{\alpha} + \til{\beta} = \alpha$.
\end{lem}

\begin{proof}
 Take $\delta \in I$ adjacent to $\alpha$. Then $\til{\beta} := \delta + \alpha \in \Phi^+$ and the claim follows for $\til{\alpha} := -\delta$.
\end{proof}

\begin{lem} \label{derTrick}
 Let $\mf{X}_\alpha, \mf{X}_\beta \leq \uCD(R)$ be distinct root subgroups with $\beta \neq -\alpha$. There exist a one-dimensional subtorus $H_{\alpha,\beta}(R) \leq \H$, say given by $h: \Mult(R) \xrightarrow{\cong} H_{\alpha, \beta},~ u \in R^\times \mapsto h(u)$, and an integer $n = n(\alpha,\beta) \neq 0$ such that $H_{\alpha,\beta}(R)$ centralizes $\mf{X}_\alpha$ and $h(u) x_\beta(r) h(u)^{-1} = x_\beta(u^n r)$ for all $x_\beta(r) \in \mf{X}_\beta$ and all $h(u) \in H_{\alpha, \beta}(R)$.
\end{lem}

\begin{proof}
Without loss of generality we may assume $\alpha$ and $\beta$ to be simple via the action of the Weyl group. If $\alpha$ is orthogonal to $\beta$, define $H_{\alpha, \beta}$ to be the semi-simple root subgroup $\mc{H}_\beta$ and the claim follows. If $\rk(\Phi) \geq 3$, this is also easy to achieve: Choose another simple root $\gamma$ which is adjacent to $\beta$ and non-adjacent to $\alpha$ and set $H_{\alpha, \beta}(R) = \mc{H}_\gamma$. For the general case, pick integers $p, q \in \Z \backslash \set{0}$ such that $2p - q\cdot (\alpha, \beta) = 0$ and set $h(u) := h_\beta(u)^{-q} h_\alpha(u)^p$ and $H_{\alpha, \beta}(R) := \gera{\set{h(u) \tq u \in R^\times}} \leq \mc{H}$. By Steinberg's relations \bref{SteinbergRel}, we have that $H_{\alpha, \beta}(R)$ centralizes $\mf{X}_\alpha$. A simple computation shows that $n(\alpha, \beta) := p\cdot (\beta, \alpha) - 2q \neq 0$, and the result follows again from \bref{SteinbergRel}.
\end{proof}

From the proof of Lemma \ref{derTrick} one can see that, in many cases, $n(\alpha, \beta)$ can be taken to be $\pm 1$ or $\pm 2$, though this needs not occur in general. Furthermore, the torus $H_{\alpha, \beta}(R)$ need not be unique, so the integers $n(\alpha, \beta)$ may vary. Clearly $n(-\alpha, \beta) = -n(\alpha, \beta)$.

\begin{dfn} \label{toralcte}
 Given a base ring $R$ and two roots $\alpha, \beta \in \Phi$, we define their \emph{toral constant} to be 
 \[
 c_{\alpha, \beta}(R) = \min \set{ | n(\alpha, \beta) | \tqalt H_{\alpha, \beta}(R) \mbox{ and } n(\alpha, \beta) \mbox{ are as in Lemma \ref{derTrick}}}. 
 \]
 The \emph{toral constant of} $\Phi$ \emph{and} $R$ is defined as $c_\Phi(R) = \max \set{c_{\alpha, \beta}(R) \tq \alpha, \beta \in \Phi}$.
\end{dfn}

The toral constant shall be used soon in order to define an appropriate (finite) generating set for $R$ as a ring. We now proceed with some notation and remarks necessary to construct our presentations.

Fix $A \subseteq \Mult(R) = R^\times$ a generating set for the multiplicative (abelian) group of units of $R$. Since the parabolics of $\ueCD(R)$ are finitely generated, so are its torus $\H$ and its Borel subgroup $\B_\Phi(R)$ and we may thus assume that $A$ is finite, say $A = \set{v_1, \ldots, v_\xi}$. Looking at the action of $\H$ on each root subgroup \bref{SteinbergRel} and from the finiteness of $A$, we conclude that the Borel subgroup in type $A_1$, i.e. $\Bzero \leq E_{A_1}^{sc}(R)$, must also be finitely generated. Now, $\Bzero$ is isomorphic to the semi-direct product $\Addi(R) \rtimes \Mult(R)$, where the action is given by
\[
\xymatrix@R=2mm{
 \Mult(R) \times \Addi(R) \ar[r] & \Addi(R) \\
 R^\times \times R \ni (u,r) \ar@{|->}[r] & u^2r.
}
\]
We may therefore choose a finite set $T_0 = \set{x_0 = 1, x_1, \ldots, x_\nu} \subseteq R$ such that $R^\times \cdot T_0 := \set{ux_i \tq 0 \leq i \leq \nu \mbox{ and } u \in \gera{A}}$ additively generates $R$. Given a positive integer $c \in \N$, let $A^{[c]}$ denote the set of monomials $\set{ v_1^{\veps_1} \cdots v_\xi^{\veps_\xi} \tq -c \leq \veps_j \leq c ~\forall j}$ over $A$ with powers of the letters bounded by $\pm c$ (notice that $1 \in A^{[c]}$). Using the action of $R^\times$ on $R$ given above, we have that $A^{[2]}\cdot T_0$ generates $R$ as a $\Z[R^\times]$-module. Setting $c_\Phi := c_\Phi(R) \in \N$ the toral constant of $\Phi$ and $R$, one still has that $\til{T} := A^{[c_\Phi]}\cdot T_0$ is a generating set for $R$ as a $\Z[R^\times]$-module. Hence, the set $T := \gera{A^{[c_\Phi]}} \cdot T_0$ additively generates the ring $R$.

Using the notation above, for every pair $x_i^m, x_j^n$ with $x_i, x_j \in T_0$ and $m, n \in \N$, we fix an expression $p(x_i^m,x_j^n) = \sum_{l=0}^\nu a_l v_l x_l$ in terms of $T$, where the $v_l \in \gera{A^{[c_\Phi]}}$ are \emph{uniquely} determined and $a_l \in \Z$, for the product $x_i^m x_j^n$. That is, the image of $p(x_i^m, x_j^n)$ in $R$ equals the product $x_i^m x_j^n$. Since $\gera{A^{[c_\Phi]}} \cdot T_0$ additively generates $R$ we can, given $r, s \in R$, extend the product $p(x_i^m,x_j^n)$ distributively and $\Z[R^\times]$-linearly in order to obtain a unique expression $p(r^m,s^n)$ in terms of $T$ for $r^m s^n \in R$ after decompositions of both $r$ and $s$ in terms of $T$. In particular, this product map $p(\cdot,\cdot)$ has the following properties.
\begin{equation} \label{produtoehmulti}
u \cdot p(r^m, s^n) = p(u \cdot r^m, s^n) = p(r^m, u \cdot s^n) \mbox{ for all } u \in \gera{A^{[c_\Phi]}},
\end{equation}
\begin{equation} \label{produtoehaddi}
p(r,s) = p(r_0, s) + p(r', s) = p(r, s_0) + p(r, s')
\end{equation}
whenever $r = r_0 + r'$ and $s = s_0 + s'$. Because $\gera{A^{[c_\Phi]}} = R^\times$ is a finitely generated abelian group, we can furthermore decompose the units $v_l$ occurring in $p$ above \emph{uniquely} as $v_l = w_l^{2k_l}u_l$ for some $w_l \in \gera{A^{[c_\Phi]}}, k_l \in \Z$ and $u_l \in A^{[c_\Phi]}$. In this way, $p(r^m, s^n)$ becomes
\begin{equation} \label{produtogeradores}
p(r^m, s^n) = \sum_{l=0}^\nu a_l w_l^{2k_l} u_l x_l, \mbox{ where } a_l \in \Z, w_l \in \gera{A^{[c_\Phi]}}, u_l x_l \in \til{T}.
\end{equation}
For a pair of roots $\gamma, \eta \in \Phi$, powers $r^m, s^n$ and the fixed expression $p(r^m, s^n)$ above, we define the formal expression
\begin{equation} \label{produtaozao}
\zeta(\gamma,\eta,r^m, s^n) = \prod_{l = 0}^\nu h_{m\gamma+n\eta}(w_l)^{k_l} x_{m\gamma + n\eta}(u_l x_l)^{a_l C^{\gamma,\eta}_{m,n}} h_{m\gamma+n\eta}(w_l)^{-k_l},
\end{equation}
the $C^{\gamma,\eta}_{m,n}$ being the structure constants from the commutator formula \bref{ChevalleyRel}.

In what follows we shall make repeated use of the following easy commutator identity.

\begin{lem} \label{comutinho}
Let $G$ be a group and let $a,b,c \in G$. Then
\[
[ab,c] = a[b,c]a^{-1}[a,c].
\]
\end{lem}

The next observation is our key starting point. When we construct our finite presentations, it will allow us to discard most of the commutator relations occurring in the unipotent radical of $\P_I(R)$.

\begin{lem} \label{DERTrick}
 Let $R, T$ and $\til{T}$ be as above, let $\gamma_0, \eta_0 \in \Phi$ be distinct roots of the same sign and let $\mf{X}_{\gamma_0} = \gera{ \set{x_{\gamma_0}(r) \tq r \in R}}$ and $\mf{X}_{\eta_0} = \gera{\set{x_{\eta_0}(s) \tq s \in R}}$ be the corresponding unipotent root subgroups in $\uCD(R)$. Then the Chevalley commutator formula
 \[
[x_\gamma(r), x_\eta(s)] = \begin{cases}\displaystyle \underset{m, n > 0}{\prod_{m\gamma + n\eta \in \Phi}} \left( x_{m\gamma+n\eta}(r^m s^n)^{C^{\gamma,\eta}_{m,n}} \right), & \mbox{if } \gamma + \eta \in \Phi;\\ 1 & \mbox{otherwise}, \end{cases}
\]
for $\gamma, \eta \in \Phi_{\set{\gamma_0, \eta_0}}^+$ distinct, and Steinberg's conjugation relations
\[
h_\beta(u) x_\alpha(r) h_\beta(u)^{-1} = x_\alpha(u^{(\alpha, \beta)}r)
\]
 in $\gera{\H_\delta,~ \mf{X}_\delta \tqalt \delta \in \Phi_{\set{\gamma_0, \eta_0}}^+} \leq \uCD(R)$ are consequences of the following relations. \\

For all $\alpha \in \Phi_{\set{\gamma_0, \eta_0}}^+$,
\begin{equation} \label{defrelstoro}
\prod_{i=1}^{\xi} h_\alpha(v_i)^{\veps_i} = 1
\end{equation}
for each arbitrary, but fixed, defining relator $v_1^{\veps_1} \cdots v_\xi^{\veps_\xi} = 1$ of $\gera{A^{[c_\Phi]}}$ as an abelian group;

For all $u, v \in A$ and $\alpha, \beta \in \Phi_{\gamma_0, \eta_0}^+$,
\begin{equation} \label{toroehabeliano}
 h_\alpha(v) \mbox{ and } h_\beta(u) \mbox{ commute};
\end{equation}

For all $v \in A,~ u x_i \in \til{T}$ and $\alpha, \beta, \gamma \in \Phi_{\set{\gamma_0, \eta_0}}^+$,
\begin{equation} \label{TrickT}
 h_\alpha(v) x_\gamma(u x_i) h_\alpha(v)^{-1} = h_\beta(v)^k x_\gamma(u' x_i) h_\beta(v)^k,
\end{equation}
where $k \in \Z$ and $u' \in A^{[c_\Phi]}$ are unique such that $v^{(\gamma,\alpha)}u = v^{k(\gamma,\beta)}u'$;

For all $r, s \in R$ and $\gamma \in \Phi_{\set{\gamma_0, \eta_0}}^+$,
\begin{equation} \label{faltouabeliano}
x_\gamma(r) \mbox{ and } x_\gamma(s) \mbox{ commute};
\end{equation}

For all $t_1, t_2 \in \til{T}$,
 \begin{equation} \label{TrickU}
[x_\gamma(t_1), x_\eta(t_2)] =
\begin{cases}
\displaystyle \underset{m\gamma + n\eta \in \Phi}{\prod_{m, n > 0}} \zeta(\gamma,\eta,t_1^m, t_2^n), & \mbox{if } \gamma + \eta \in \Phi; \\
1 & \mbox{otherwise},
\end{cases}
\end{equation}
where $p(t_1^m, t_2^n)$ is a fixed expression for the product $t_1^m t_2^n$ in terms of $T$ as described in \bref{produtogeradores} and $\zeta$ is as in \bref{produtaozao}.
\end{lem}

\begin{proof}
First of all, some clarification. To say that the commutator formula and conjugation relations follow from the relations given above means that, if one can write the $h_\delta(u),~ u \in R^\times$, and $x_\delta(r),~ r \in R$ and $\delta \in \Phi_{\set{\gamma_0, \eta_0}}^+$, as (appropriate) products of elements for which the given relations \bref{defrelstoro} -- \bref{TrickU} hold, then the commutator formula and the conjugation relations are in fact formal consequences of \bref{defrelstoro} -- \bref{TrickU}. From the previous discussion on the generating set for $R$, we know that every element $r \in R$ decomposes as
\[
r = \sum_{l = 0}^{\nu}a_l w_l^{2k_l} u_l x_l, \mbox{ where } a_l \in \Z, w_l \in \gera{A^{[c_\Phi]}}, u_l x_l \in \til{T}.
\]
We may then write
\[
x_\gamma(r) = \prod_{l = 0}^{\nu} h_\gamma(w_l)^{k_l} x_\gamma(u_l x_l)^{a_l} h_\gamma(w_l)^{-k_l}.
\]
If furthermore $w_l = v_1^{n_1} \cdots v_\xi^{n_\xi}$, then $h_\gamma(w_l) = h_\gamma(v_1)^{n_1} \cdots h_\gamma(v_\xi)^{n_\xi}$. We now break the proof into several steps.

\underline{Step 0.} Steinberg's relations hold. Moreover,
\begin{equation} \label{produtonotoro}
\left(\prod_{i=1}^n h_{\beta_i}(u_i)^{\veps_i}\right) x_\alpha(r) \left(\prod_{i=1}^n h_{\beta_i}(u_i)^{\veps_i}\right)^{-1} = 
x_\alpha\left(r\prod_{i=1}^n u_i^{\veps_i \cdot (\alpha,\beta_i)}\right).
\end{equation}
Indeed, the relations \bref{defrelstoro} and \bref{toroehabeliano} together imply that each subgroup $\H_\alpha = \langle \{ h_\alpha(v) \tq v \in A \} \rangle \leq \langle \H_\delta,~ \mf{X}_\delta \tqalt \delta \in \Phi_{\set{\gamma_0, \eta_0}}^+ \rangle$ is isomorphic to $R^\times$ and the whole torus $\langle \H_\delta \tqalt \delta \in \Phi_{\set{\gamma_0, \eta_0}}^+ \rangle$ is in fact abelian. This combined with \bref{TrickT} yield Steinberg's relations. Equation \bref{produtonotoro} is just an iterated application of said relations. In particular, the following identity holds.
\begin{equation} \label{produtonotoro2}
\left(\prod_{i=1}^n h_{\beta}(u_i)^{\veps_i}\right) x_\alpha(r) \left(\prod_{i=1}^n h_{\beta}(u_i)^{\veps_i}\right)^{-1} = 
\end{equation}
\[
h_\beta\left(\prod_{i=1}^n u_i^{\veps_i}\right) x_\alpha(r) h_\beta\left(\prod_{i=1}^n u_i^{\veps_i}\right)^{-1} =
\]
\[
x_\alpha\left(r\prod_{i=1}^n u_i^{\veps_i \cdot (\alpha,\beta)}\right).
\]

\underline{Step 1.} Let $r = w^{2k} u x_i, s = z^{2l} v x_j \in T$. Then the commutator formula holds. 

To prove this, we identify $x_\gamma(r) = h_\gamma(w)^k x_\gamma(u x_i) h_\gamma(w)^{-k}$ and $x_\eta(s) = h_\eta(z)^l x_\eta(v x_j) h_\eta(z)^{-l}$. By Lemma \ref{derTrick}, there exist roots $\alpha, \beta \in \Phi_{\set{\gamma_0, \eta_0}}^+$ such that $(\gamma, \alpha) \neq 0 \neq (\eta, \beta)$ and $(\eta, \alpha) = 0 = (\gamma, \beta)$. Choose $u_0, v_0 \in A^{[c_\Phi]}$ and $k_0, l_0 \in \Z$ such that $w^{2k} u = w^{k_0(\gamma,\alpha)}$ and $z^{2l} v=z^{l_0(\eta,\beta)} v_0$. By \bref{TrickT},
\[
x_\gamma(r) = h_\alpha(w)^{k_0} x_\gamma(u_0 x_i) h_\alpha(w)^{-k_0} \mbox{ and } x_\eta(s) = h_\beta(z)^{l_0} x_\eta(v_0 x_j) h_\beta(z)^{-l_0}.
\]
Write $p(x_i^m, x_j^n) = \sum_{e = 0}^\nu a_e w_e^{2k_e} u_e x_e$ as in \bref{produtogeradores}. By property \bref{produtoehmulti} we obtain
\[
p(r^m, s^n) = p((w^{2k} u x_i)^m, (z^{2l} v x_j)^n) = \sum_{e = 0}^\nu a_e (w^{2km} u^m z^{2ln} v^n w_e^{2k_e} u_e) x_e
\]
and
\[
p((u_0 x_i)^m, (v_0 x_j)^n) = \sum_{e = 0}^\nu a_e (u_0^m v_0^n w_e^{2k_e} u_e) x_e.
\]
For every $e$ let $\til{z}_e, \til{w}_e \in \gera{A}^{[c_\Phi]},~ \til{v}_e, \til{u}_e \in A^{[c_\Phi]}$ and $\til{l}_e, \til{k}_e \in \Z$ be unique such that
\[
w^{2km} u^m z^{2ln} v^n w_e^{2k_e} u_e = \til{z}_e^{2\til{l}_e} \til{v}_e \mbox{ and } u_0^m v_0^n w_e^{2k_e} u_e = \til{w}_e^{2\til{k}_e} \til{u}_e.
\]
By definition of the expression \bref{produtaozao} we have, on the one hand,
\[
[x_\gamma(r), x_\eta(s)] =
\]
\[
\begin{cases}
\displaystyle \underset{m, n > 0}{\prod_{m\gamma + n\eta \in \Phi}} \prod_{e=0}^\nu h_{m\gamma + n\eta}(\til{z}_e)^{\til{l}_e} x_{m\gamma + n\eta}(\til{v}_e x_e)^{a_e \cdot C_{m,n}^{\gamma, \eta}} h_{m\gamma + n\eta}(\til{z}_e)^{-\til{l}_e}, \mbox{ if } \gamma + \eta \in \Phi; \\
1 \mbox{ otherwise}.
\end{cases}
\]
On the other hand,
\[
[x_\gamma(r),x_\eta(s)] = h_\alpha(w)^{k_0}h_\beta(z)^{l_0} [x_\gamma(u_0 x_i), x_\eta(v_0 x_j)] h_\beta(z)^{-l_0}h_\alpha(w)^{-k_0} =
\]
\[
\begin{cases}
\displaystyle \underset{m\gamma + n\eta \in \Phi}{\prod_{m, n > 0}} h_\alpha(w)^{k_0}h_\beta(z)^{l_0} \zeta(\gamma,\eta,(u_0 x_i)^m, (v_0 x_j)^n) h_\beta(z)^{-l_0}h_\alpha(w)^{-k_0}, \\
 \mbox{if } \gamma + \eta \in \Phi; \\
1 \mbox{ otherwise},
\end{cases}
\]
and
\[
\zeta(\gamma,\eta,(u_0 x_i)^m, (v_0 x_j)^n) =
\]
\[
\displaystyle \prod_{e=0}^\nu h_{m\gamma + n\eta}(\til{w}_e)^{\til{k}_e} x_{m\gamma + n\eta}(\til{u}_e x_e)^{a_e \cdot C_{m,n}^{\gamma, \eta}} h_{m\gamma + n\eta}(\til{w}_e)^{-\til{k}_e}.
\]
Pick $\til{u}'_e, \til{u}''_e \in A^{[c_\Phi]}$ and $l'_0, k'_0 \in \Z$ such that $z^{l_0 \cdot (\eta, \beta) \cdot n} \til{u}_e = z^{2 l_0'} \til{u}_e'$ and $w^{k_0 \cdot (\gamma, \alpha) \cdot m} \til{u}_e' = w^{2 k_0'} \til{u}_e''$. By \bref{TrickT}, \bref{produtonotoro} and \bref{produtonotoro2}, the equation above becomes
\[
h_\alpha(w)^{k_0}h_\beta(z)^{l_0} [x_\gamma(u_0 x_i), x_\eta(v_0 x_j)] h_\beta(z)^{-l_0}h_\alpha(w)^{-k_0} =
\]
\[
\begin{cases}
\displaystyle \underset{m\gamma + n\eta \in \Phi}{\prod_{m, n > 0}} \prod_{e=0}^\nu h_{m\gamma + n\eta}(\til{w}_e^{\til{k}_e} z^{2 l_0'} w^{2 k_0'}) x_{m\gamma + n\eta}(\til{u}_e'' x_e)^{a_e \cdot C_{m,n}^{\gamma, \eta}} h_{m\gamma + n\eta}(\til{w}_e^{\til{k}_e} z^{2 l_0'} w^{2 k_0'})^{-1}, \\
\mbox{if } \gamma + \eta \in \Phi; \\
1 \mbox{ otherwise},
\end{cases}
\]
But
\begin{align*}
\til{w}_e^{2\til{k}_e} z^{2 l_0'} w^{2 k_0'} \til{u}_e'' & = \til{w}_e^{2\til{k}_e} w^{k_0 \cdot (\gamma, \alpha) \cdot m} z^{2 l_0'} \til{u}_e'\\
& = w^{k_0 \cdot (\gamma, \alpha) \cdot m} z^{l_0 \cdot (\eta, \beta)\cdot n} \til{w}_e^{2\til{k}_e} \til{u}_e\\
& = w^{k_0 \cdot (\gamma, \alpha) \cdot m} u^m z^{l_0 \cdot (\eta, \beta)\cdot n} v^n w^{2k_e} u_e\\
& = \til{z}^{2 \til{l}_e} \til{v}_e,
\end{align*}
so the claim follows from \bref{TrickT} and \bref{produtonotoro}.

\underline{Step 2.} The commutator relation holds whenever $\gamma + \eta \notin \Phi$.

Write $r = \sum_{l=0}^\nu a_l w_l^{2k_l} u_l x_l$ and $s = \sum_{l=0}^\nu b_l z_l^{2j_l} v_l x_l$ and let $N := \sum_{l=0}^\nu (|a_l|+|b_l|)$. We proceed by induction on $N$. We observe that Step 1 gives the base case. We have that
\[
[x_\gamma(r),x_\eta(s)] =
\]
\[
= \left[h_\gamma(w_0)^{k_0} x_\gamma(u_0 x_0)^{a_0} h_\gamma(w_0)^{-k_0} \prod_{l=1}^\nu h_\gamma(w_l)^{k_l} x_\gamma(u_l x_l)^{a_l} h_\gamma(w_l)^{-k_l} \right.
\]
\[
\left., \prod_{l=0}^\nu h_\eta(z_l)^{j_l} x_\eta(v_l x_l)^{b_l} h_\eta(z_l)^{-j_l}\right] =
\]
\begin{align*}
h_\gamma(w_0)^{k_0} x_\gamma(u_0 x_0)^{a_0} h_\gamma(w_0)^{-k_0} \left[ \prod_{l=1}^\nu h_\gamma(w_l)^{k_l} x_\gamma(u_l x_l)^{a_l} h_\gamma(w_l)^{-k_l}, \right. & \\
\left. \prod_{l=0}^\nu h_\eta(z_l)^{j_l} x_\eta(v_l x_l)^{b_l} h_\eta(z_l)^{-j_l} \right] h_\gamma(w_0)^{k_0} x_\gamma(u_0 x_0)^{-a_0} h_\gamma(w_0)^{-k_0} \times &
\end{align*}
\[
\times \left[ h_\gamma(w_0)^{k_0} x_\gamma(u_0 x_0)^{a_0} h_\gamma(w_0)^{-k_0}, \prod_{l=0}^\nu h_\eta(z_l)^{j_l} x_\eta(v_l x_l)^{b_l} h_\eta(z_l)^{-j_l} \right].
\]
By induction hypothesis, we have commutator relations for the terms in the commutators above, which in the case $\gamma + \eta \notin \Phi$ mean that all commutators above vanish, thus proving the claim.

\underline{Step 3.} The commutator relations hold whenever the height $ht(\gamma+\eta)$ is maximal in $\Phi_{\set{\gamma_0, \eta_0}}^+$.

This is very similar to the previous case. Recall that the height of a root $\delta = a \gamma_0 + b \eta_0 \in \Phi_{\set{\gamma_0, \eta_0}}^+$ is the integer $a + b$. Take expressions for $r$ and $s$ as before and define $N$ in the same way. Again, Step 1 gives the base case. Proceed as before to obtain the exact last expression above. By induction hypothesis, we have commutator relations for the terms in the commutators above, so the considered expression becomes
\[
h_\gamma(w_0)^{k_0} x_\gamma(u_0 x_0)^{a_0} h_\gamma(w_0)^{-k_0} x_{\gamma+\eta}(r' s) \times 
\]
\[
\times h_\gamma(w_0)^{k_0} x_\gamma(u_0 x_0)^{-a_0} h_\gamma(w_0)^{-k_0} x_{\gamma + \eta}(u_0 x_0 s),
\]
where $r' \in R$ is represented (via $p$ and $\zeta$) by the double product in the left hand side commutator. Now, by Lemma \ref{derTrick} and \bref{TrickT}, we may rewrite
\[
h_\gamma(w_0)^{k_0} x_\gamma(u_0 x_0)^{a_0} h_\gamma(w_0)^{-k_0} = h_\alpha(w_0)^{k'_0} x_\gamma(u'_0 x_0)^{a_0} h_\alpha(w_0)^{-k'_0},
\]
where $(\gamma + \eta, \alpha) = 0$. In particular, $h_\alpha(w_0)$ commutes with $x_{\gamma + \eta}(r' s)$. But, since $ht(\gamma + \eta)$ is maximal, $x_\gamma(u'_0 x_0)$ also commutes with $x_{\gamma + \eta}(r' s)$. Thus, $h_\gamma(w_0)^{k_0} x_\gamma(u_0 x_0)^{a_0} h_\gamma(w_0)^{-k_0}$ vanishes from the expression above and the claim follows by linearly expanding the expression $p$ for the products above and repeatedly applying \bref{produtoehmulti}, \bref{produtoehaddi} together with \bref{faltouabeliano}.

\underline{Step 4.} We can now finish the proof of the lemma. Notice that Step 3 implies the result for simply-laced root systems. Suppose $\gamma + \eta \in \Phi$ and $ht(\gamma + \eta)$ is not maximal. Again taking expressions for $r$ and $s$ and defining $N$ as in Step 2, proceed by induction on $N$ and reverse induction on $ht(\gamma + \eta)$, that is, we first prove the result on roots of maximal height and then descend to minimal roots. Without loss of generality, assume $\gamma$ is long and $\eta$ is short. The base case consists of Steps 1 and 3. Assuming the commutator relations hold for roots of all heights at least $ht(\gamma + \eta)$, we may proceed analogously as in the last part of Step 3 because $\gamma$ is long. The lemma follows.
\end{proof}

Apart from the remaining commutator relations, we will have to deal with the additive relations coming from $R$. We chose $\til{T} = A^{[c_\Phi]}\cdot T_0$ as generating set for $R$ as a ring and $T = \gera{A^{[c_\Phi]}}\cdot T_0$ as the corresponding generating set for $\Addi(R)$, so we may fix $\mc{A}$ an arbitrary set of \emph{additive defining relators} for $R$ such that every expression $a \in \mc{A}$ is of the form
\[
a = \sum_{l = 0}^{\nu}a_l w_l^{2k_l} u_l x_l, \mbox{ where } a_l \in \Z, w_l \in \gera{A^{[c_\Phi]}}, u_l x_l \in \til{T}.
\]
In other words, $\mc{A}$ is a fixed subset of $\bigoplus_{x \in T_0} \Z[\gera{A^{[c_\Phi]}}]\cdot x$, with elements given in the form above, and with the property that
\[
\Addi(R) \cong \displaystyle\frac{\bigoplus_{x \in T_0} \Z[\gera{A^{[c_\Phi]}}]\cdot x}{\vspan(\mc{A})}.
\]

\subsection{Proof of Theorem \texorpdfstring{\ref{avolta}}{4.1} for \texorpdfstring{$I = \vazio$}{I empty}} \label{provaBorel} Recall that the Levi decomposition for $\mcB_\Phi(R)$ is just $\mcB_\Phi(R) = \U_\vazio(R) \rtimes \mc{H}$, where $\mc{H} = \gera{\mc{H}_\alpha \tqalt \alpha \in \Phi}$ is the standard torus and $\U_\vazio(R) = \gera{\mf{X}_\gamma \tqalt \gamma \in \Phi^+}$.

Consider the following sets of relations. For all $\alpha \in \Phi, \gamma \in \Phi^+, v \in A, t \in T$,
\begin{equation} \label{BorelWT}
 h_\alpha(v) x_\gamma(t) h_\alpha(v)^{-1} = x_\gamma(v^{(\gamma,\alpha)}t).
\end{equation}
For all $t_1, t_2 \in T, \gamma, \eta \in \Phi^+,$
 \begin{equation} \label{BorelU}
[x_\gamma(t_1), x_\eta(t_2)] =
\begin{cases}
\displaystyle \prod_{m\gamma + n\eta \in \Phi^+} \zeta(\gamma,\eta,t_1^m, t_2^n), & \mbox{if } \gamma + \eta \in \Phi; \\
1 & \mbox{otherwise}, \end{cases}
\end{equation}
where $p(t_1^m, t_2^n)$ is a fixed expression for the product $t_1^m t_2^n$ in terms of $\gera{A^{[c_\Phi]}}$ and $\til{T}$ as described in \bref{produtogeradores} and $\zeta$ is as in \bref{produtaozao}.\\
For all $a = \sum_{l = 0}^{\nu}a_l w_l^{2k_l} u_l x_l \in \mc{A}$ and $\gamma \in \Phi^+$,
\begin{equation} \label{BorelADD}
\prod_{l = 0}^{\nu} x_\gamma(w_l^{2k_l} u_l x_l)^{a_l} = 1.
\end{equation}
Let $\mc{S}_\mcB$ be the set of all relations \bref{BorelWT}, \bref{BorelU} and \bref{BorelADD} given above.

Since the torus $\mc{H}$ is a finitely generated abelian group, we may fix a presentation
\[
\mc{H} \cong \gera{\set{h_\alpha(v) \tq \alpha \in \Phi, v \in A} \tq \mc{T}},
\]
where $\mc{T}$ is finite. Combining this with Lemma \ref{presunip} and the given descriptions of $R$, $\mc{A}$ and $\mc{S}_\mcB$, we obtain a presentation
\begin{equation} \label{standardpresBorel}
\mcB_\Phi(R) \cong \gera{\set{h_\alpha(v), x_\gamma(t) \tq \alpha \in \Phi, \gamma \in \Phi^+, v \in A, t \in T} \tq \mc{T} \cup \mc{S}_\mcB}.
\end{equation}

Suppose the extended Levi factors $\LE_n(R)$ of $\B_\Phi(R)$ are finitely presented. By definition, they consist just of a single unipotent root subgroup acted upon non-trivially by the standard torus. The point now is that, in fact, every subgroup $\mf{X}_\gamma \rtimes \H \leq \B_\Phi(R),~ \gamma \in \Phi^+$, is finitely presented, not just the $\LE_n(R) = \mf{X}_{\alpha_n} \rtimes \H$. This is due to the following lemma, which in turn is an easy consequence of \bref{SteinbergReflection}.

\begin{lem} \label{XHF2}
Let $\gamma$ be a positive root. Then there exist a simple root $\alpha \in \Delta$ and an isomorphism $w : \mf{X}_\gamma \rtimes \H \xrightarrow{\cong} \mf{X}_\alpha \rtimes \H$.
\end{lem}

\begin{proof}
Due to the action of the Weyl group $W$ of $\Phi$, we can find a root $\beta$ and a simple root $\alpha$ such that $r_\beta(\gamma) = \alpha$, where $r_\beta \in W$ is the reflection associated to $\beta$. Let $w$ be conjugation by $w_\beta$, where $w_\beta \in \uCD(R)$ is as in the end of Section \ref{jargao}. The result follows from \bref{SteinbergReflection}.
\end{proof}

For each $\gamma \in \Phi^+$, let then $\mf{X}_\gamma \rtimes \H = \gera{\mc{X}_\gamma \tq \mc{S}_\gamma}$ be a finite presentation with generating set
\[
\mc{X}_\gamma = \set{h_{\alpha}(v), x_\gamma(t) \tq v \in A, t \in \til{T}, \alpha \in \Phi}.
\]
Let $\til{\mc{Y}}$ be the \emph{finite} set of generators
\[
\til{\mc{Y}} = \set{\til{x}_\gamma(t) \tq t\in \til{T}, \gamma \in \Phi^+}.
\]
We define the following \emph{finite} sets of relations. For all $v \in A, ux_i \in \til{T}, \alpha, \beta \in \Phi, \gamma \in \Phi^+$,
\begin{equation} \label{tilBorelWT}
 h_\alpha(v) \til{x}_\gamma(ux_i) h_\alpha(v)^{-1} = h_\beta(v)^k \til{x}_\gamma(u'x_i) h_\beta(v)^{-k},
\end{equation}
where $k \in \Z$ and $u' \in A^{[c_\Phi]}$ are unique such that $v^{(\gamma,\alpha)}u = v^{k(\gamma,\beta)}u'$.\\
For all $t_1, t_2 \in \til{T}, \gamma, \eta \in \Phi^+,$
 \begin{equation} \label{tilBorelU}
[\til{x}_\gamma(t_1), \til{x}_\eta(t_2)] =
\begin{cases}
\displaystyle \prod_{m\gamma + n\eta \in \Phi^+} \til{\zeta}(\gamma,\eta,t_1^m, t_2^n), & \mbox{if } \gamma + \eta \in \Phi; \\
1 & \mbox{otherwise}, \end{cases}
\end{equation}
where $p(t_1^m, t_2^n)$ is a fixed expression for the product $t_1^m t_2^n$ in terms of $\gera{A^{[c_\Phi]}}$ and $\til{T}$ as in \bref{produtogeradores} and $\til{\zeta}$ is obtained from $\zeta$ in \bref{produtaozao} by formally replacing $x_{m\gamma + n\eta}$ by $\til{x}_{m\gamma + n \eta}$.\\
Finally, let $\til{\mc{S}}_\mcB$ be the union of the sets $\mc{S}_\gamma$ (with $\gamma$ running over $\Phi^+$) and the sets of all relations \bref{tilBorelWT}, \bref{tilBorelU} given above. Notice that we did not add any defining relators coming from the underlying additive group $\Addi(R)$, except for those possibly contained in the $\mc{S}_\gamma$.

Let $\til{\mcB}_\Phi(R)$ be the group given by the presentation
\begin{equation} \label{prestilBorel}
\til{\mcB}_\Phi(R) \cong \gera{\set{h_\alpha(v) \tq \alpha \in \Phi, v \in A} \cup \til{\mc{Y}} \tq \mc{T} \cup \til{\mc{S}}_\mcB}.
\end{equation}
By construction, $\til{\mcB}_\Phi(R)$ is finitely presented. We claim that $\til{\mcB}_\Phi(R) \cong \mcB_\Phi(R)$. Consider the map $h_\alpha(v) \mapsto h_\alpha(v), \til{x}_\gamma(t) \mapsto x_\gamma(t)$ from $\til{\mcB}_\Phi(R)$ to $\mcB_\Phi(R)$. Since (the images of) all the relations $\mc{T} \cup \til{\mc{S}}_\mcB$ hold in $\mcB_\Phi(R)$ and the latter is generated by $\mc{H}$ and the $x_\gamma(t), t \in \til{T}$, we get a natural epimorphism $\til{\mcB}_\Phi(R) \onto \mcB_\Phi(R)$ by von Dyck's theorem. To prove that this is also injective, let $F$ be the free group on the generating set $\set{h_\alpha(v), x_\gamma(t) \tq \alpha \in \Phi, \gamma \in \Phi^+, v \in A, t \in T}$ of \bref{standardpresBorel} and consider the homomorphism $f$ given by
\[
\xymatrix@R=2mm{
f : F \ar[r] & \til{\B}_\Phi(R) \\
h_\alpha(v) \ar@{|->}[r] & h_\alpha(v) \\
x_\gamma(t) \ar@{|->}[r] & h_\gamma(w) \til{x}_\gamma(ux_i) h_\gamma(w)^{-1},
}
\]
where $w \in \gera{A^{[c_\Phi]}}$ and $ux_i \in \til{T}$ are unique such that $t = w^2 u x_i$. It suffices to show that the set of relations $\mc{S}_\mcB$ given in \bref{standardpresBorel} is contained in $\ker(f)$.

We first consider the relations \bref{BorelWT}. Let $\alpha \in \Phi, \gamma \in \Phi^+, v \in A$ and $t = w^2 u x_i \in T$. Let $w_0 \in \gera{A^{[c_\Phi]}}$ and $u' \in A^{[c_\Phi]}$ be unique such that $v^{(\gamma, \alpha)} w^2 u = w_0^2 u'$, and choose further $k \in \Z$ and $u'' \in A^{[c_\Phi]}$ unique such that $v^{(\gamma, \alpha)} u = v^{2k} u''$. By \bref{tilBorelWT}, we obtain
\[
f(h_\alpha(v)x_\gamma(t)h_\alpha(v)^{-1}x_\gamma(v^{(\gamma,\alpha)}t)^{-1}) =
\]
\[
= h_\gamma(w) h_\alpha(v) \til{x}_\gamma(u x_i) h_\alpha(v)^{-1} h_\gamma(w)^{-1} h_\gamma(w_0) \til{x}_(u'x_i)^{-1} h_\gamma(w_0)^{-1}
\]
\[
= h_\gamma(w) h_\gamma(v)^k \til{x}_\gamma(u'' x_i) h_\gamma(v)^{-k} h_\gamma(w)^{-1} h_\gamma(w_0) \til{x}_\gamma(u'x_i)^{-1} h_\gamma(w_0)^{-1}.
\]
But
\[
v^{2k} w^2 u'' = v^{(\gamma, \alpha)} w^2 u = w_0^2 u', \mbox{ so } u'' = u', w_0 = v^k w,
\]
whence
\[
\til{x}_\gamma(u'' x_i) = \til{x}_\gamma(u'x_i) \mbox{ and } h_\gamma(w_0) = h_\gamma(v)^k h_\gamma(w),
\]
and so
\[
f(h_\alpha(v)x_\gamma(t)h_\alpha(v)^{-1}x_\gamma(v^{(\gamma,\alpha)}t)^{-1}) = 1.
\]

We now want to show that
\[
f([x_\gamma(t_1), x_\gamma(t_2)]) = 1.
\]
This is essentially trivial, for $f$ restricted to $\gera{\mf{X}_\gamma, \mc{H}} \leq \B_\Phi(R)$ yields a surjection onto 
\[
\gera{\til{\mf{X}}_\gamma, \mc{H}_\gamma} = \gera{\set{\til{x}_\gamma(t), h_\gamma(v) \tq t \in \til{T}, v \in A}}
\]
by definition of $f$ and because $\til{\mc{S}}_\B$ from \bref{prestilBorel} contains (the copy of) $\mc{S}_\gamma$. With the relations above at hand, we can now apply Lemma \ref{DERTrick} and the relations \bref{tilBorelU} and recover the remaining commutator relations. It then follows that all relations \bref{BorelU} are contained in $\ker(f)$.

Similarly, we see that the relations \bref{BorelADD} are in $\ker(f)$. Indeed, if $\gamma \in \Phi^+$ and $a = \sum_{l=0}^\nu a_l w_l^{2k_l} u_l x_l \in \mc{A}$, then $f$ maps
\[
\prod_{l = 0}^{\nu} x_\gamma(w_l^{2k_l} u_l x_l)^{a_l} = \prod_{l = 0}^{\nu} h_\gamma(w_l) x_\gamma(u_l x_l)^{a_l} h_\gamma(w_l)^{-1}
\]
to $1$ in $\til{\mcB}_\Phi(R)$, for this holds in $\gera{\mf{X}_\gamma, \mc{H}}$, which surjects onto $\gera{\til{\mf{X}}_\gamma, \mc{H}} \leq \til{\B}_\Phi(R)$, and this concludes the proof. 

\subsection{Proof of Theorem \texorpdfstring{\ref{avolta}}{4.1} for \texorpdfstring{$I \neq \vazio$}{I non-empty}} \label{provaParabs} Recall from Section \ref{Levi} that $\P_I(R) = \mc{K}_I(R) \rtimes \LE_I(R) \leq \ueCD(R)$,
\[
\mc{K}_I(R) = \gera{\mf{X}_\gamma \tqalt \gamma \in \Phi^+ \backslash \Phi_{\Ext(I)}} \mbox{ and}
\]
\[
\LE_I(R) = \gera{\H_\eta, \mf{X}_\alpha \tqalt \eta \in \Phi \mbox{ and } \alpha \in \Phi_I \cup \Phi^+_{\nonAdj(I)}}.
\]
By hypothesis, we may fix a presentation $\LE_I(R) = \gera{\mc{X} \tq \mc{R}}$ with (finite) generating set
\[
\mc{X} = \set{h_\beta(v), x_\alpha(t) \tq \beta \in \Phi, \alpha \in \Phi_I \cup \Phi^+_{\nonAdj(I)}, v \in A \mbox{ and } t \in \til{T}}
\]
and $\mc{R}$ finite. Now consider the following sets of relations. For all $\beta \in \Phi, \gamma \in \Phi^+ \backslash \Phi_{\Ext(I)}, v \in A, t \in T,$
\begin{equation} \label{parabWT}
 h_\beta(v) x_\gamma(t) h_\alpha(v)^{-1} = x_\gamma(v^{(\gamma,\beta)}t).
\end{equation}
For all $t \in \til{T}, s \in T, \alpha \in \Phi_I \cup \Phi_{\nonAdj(I)}^+, \gamma \in \Phi^+ \backslash \Phi_{\Ext(I)},$
 \begin{equation} \label{parabWU}
[x_\alpha(t), x_\gamma(s)] =
\begin{cases}
\displaystyle \underset{m\alpha + n\gamma \in \Phi}{\prod_{m, n > 0}} \zeta(\alpha,\gamma,t^m, s^n), & \mbox{if } \alpha + \gamma \in \Phi; \\
1 & \mbox{otherwise}, \end{cases}
\end{equation}
where $p(t^m, s^n)$ is a fixed expression for the product $t^m s^n$ in terms of $\gera{A^{[c_\Phi]}}$ and $\til{T}$ as described in \bref{produtogeradores} and $\zeta$ is as in \bref{produtaozao}.\\
For all $t_1, t_2 \in T, \gamma, \eta \in \Phi^+ \backslash \Phi_{\Ext(I)},$
 \begin{equation} \label{parabU}
[x_\gamma(t_1), x_\eta(t_2)] =
\begin{cases}
\displaystyle \prod_{m\gamma + n\eta \in \Phi^+} \zeta(\gamma,\eta,t_1^m, t_2^n), & \mbox{if } \gamma + \eta \in \Phi; \\
1 & \mbox{otherwise}, \end{cases}
\end{equation}
where $p(t_1^m, t_2^n)$ is a fixed expression for the product $t_1^m t_2^n$ in terms of $\gera{A^{[c_\Phi]}}$ and $\til{T}$ as described in \bref{produtogeradores} and $\zeta$ is as in \bref{produtaozao}.\\
For all $a = \sum_{l = 0}^{\nu}a_l w_l^{2k_l} u_l x_l \in \mc{A}$ and $\gamma \in \Phi^+ \backslash \Phi_{\Ext(I)}$,
\begin{equation} \label{parabADD}
\prod_{l = 0}^{\nu} x_\gamma(w_l^{2k_l} u_l x_l)^{a_l} = 1.
\end{equation}
Let $\mc{S}_I$ be the set of all relations \bref{parabWT}, \bref{parabWU}, \bref{parabU} and \bref{parabADD}. Then
\begin{equation} \label{standardpresparab}
\P_I(R) \cong \gera{\mc{X} \cup \set{x_\gamma(t) \tq \gamma \in \Phi^+ \backslash \Phi_{\Ext(I)}, t \in T} \tq \mc{R} \cup \mc{S}_I}.
\end{equation}
 We break the proof in two, according to the cases of the {\QG} condition. 

Let
\[
\til{\mc{Y}} = \set{\til{x}_\gamma(t) \tq t\in \til{T}, \gamma \in \Phi^+ \backslash \Phi_{\Ext(I)}}
\]
and define the following (finite) sets of relations. For all $v \in A, ux_i \in \til{T}, \alpha, \beta \in \Phi, \gamma \in \Phi^+ \backslash \Phi_{\Ext(I)}$,
\begin{equation} \label{tilParabWT}
 h_\alpha(v) \til{x}_\gamma(ux_i) h_\alpha(v)^{-1} = h_\beta(v)^k \til{x}_\gamma(u'x_i) h_\beta(v)^{-k},
\end{equation}
where $k \in \Z$ and $u' \in A^{[c_\Phi]}$ are unique such that $v^{(\gamma,\alpha)}u = v^{k(\gamma,\beta)}u'$.\\
For all $t, s \in \til{T}, \alpha \in \Phi_I \cup \Phi_{\nonAdj(I)}^+, \gamma \in \Phi^+ \backslash \Phi_{\Ext(I)},$
 \begin{equation} \label{tilParabWU}
[x_\alpha(t), x_\gamma(s)] =
\begin{cases}
\displaystyle \underset{m\alpha + n\gamma \in \Phi}{\prod_{m, n > 0}} \til{\zeta}(\alpha,\gamma,t^m, s^n), & \mbox{if } \alpha + \gamma \in \Phi; \\
1 & \mbox{otherwise}, \end{cases}
\end{equation}
where $p(t^m, s^n)$ is a fixed expression for the product $t^m s^n$ in terms of $\gera{A^{[c_\Phi]}}$ and $\til{T}$ as described in \bref{produtogeradores} and $\til{\zeta}$ is obtained from $\zeta$ in \bref{produtaozao} by formally replacing $x_{m\gamma + n\eta}$ by $\til{x}_{m\gamma + n \eta}$.\\
For all $t_1, t_2 \in \til{T}, \gamma, \eta \in \Phi^+ \backslash \Phi_{\Ext(I)},$
 \begin{equation} \label{tilParabU}
[\til{x}_\gamma(t_1), \til{x}_\eta(t_2)] =
\begin{cases}
\displaystyle \prod_{m\gamma + n\eta \in \Phi^+} \til{\zeta}(\gamma,\eta,t_1^m, t_2^n), & \mbox{if } \gamma + \eta \in \Phi; \\
1 & \mbox{otherwise}, \end{cases}
\end{equation}
where $p(t_1^m, t_2^n)$ is a fixed expression for the product $t_1^m t_2^n$ in terms of $\gera{A^{[c_\Phi]}}$ and $\til{T}$ as in \bref{produtogeradores} and $\til{\zeta}$ is as above.

\subsubsection*{Case 1 -- \texorpdfstring{$\mcB_2^0(R)$}{Borel subgroup of SL2} is finitely presented} \label{provaParab1} Similarly to the proof of the previous case \ref{provaBorel}, we fix a \emph{finite} presentation 
\[
\mcB_2^0(R) = \gera{\set{h_{\alpha_0}(v), x_{\alpha_0}(t) \tq v \in A, t \in \til{T}} \tq \mc{S}_0} \leq E_{A_1}^{sc}(R)
\] and, for each $\gamma \in \Phi^+ \backslash \Phi_{\Ext(I)}$, let $\mc{S}_\gamma$ be the set obtained from $\mc{S}_0$ by formally replacing $\alpha_0$ by $\gamma$. Define $\til{S}_{\mcB,I}$ as $\bigcup_{\gamma \in \Phi^+ \backslash \Phi_{\Ext(I)}} \mc{S}_\gamma$ together with the sets of all relations \bref{tilParabWT}, \bref{tilParabWU} and \bref{tilParabU}. We claim that the finitely presented group
\begin{equation}
\til{\P}_I(R) = \gera{\mc{X} \cup \til{\mc{Y}} \tq \mc{R} \cup \til{\mc{S}}_{\mcB,I}}
\end{equation}
is isomorphic to the parabolic group $\P_I(R)$.

It is clear that the natural map $h_\alpha(v) \mapsto h_\alpha(v), \til{x}_\gamma(t) \mapsto x_\gamma(t)$ from $\til{\P}_I(R)$ to $\P_I(R)$ induces an epimorphism $\til{\P}_I(R) \onto \P_I(R)$. Let $F$ be the free group on the generating set $\mc{X} \cup \set{x_\gamma(t) \tq \gamma \in \Phi^+ \backslash \Phi_{\Ext(I)}, t \in T}$ of \bref{standardpresparab} and consider the homomorphism $f : F \to \til{\P}_I(R)$ given by $x \in \mc{X} \mapsto x, x_\gamma(t) \mapsto h_\gamma(w) \til{x}_\gamma(u x_i) h_\gamma(w)^{-1}$, where $w \in \gera{A^{[c_\Phi]}}$ and $u x_i \in \til{T}$ are unique such that $t = w^2 u x_i$. We prove that $f$ induces a left-inverse of $\til{\P}_I(R) \onto \P_I(R)$ by showing that the set of relations $\mc{S}_I$ given in \bref{standardpresparab} is contained in $\ker(f)$. But this is essentially a reprise of the previous case \ref{provaBorel}.

In effect, the proof that the relations \bref{parabWT} are contained in $\ker(f)$ is exactly the one given in \ref{provaBorel}, so we won't repeat it here. Since $\til{\mc{S}}_I$ contains copies of the $\mc{S}_\gamma$ that define $\mcB_2^0(R)$, the commutativity between $\til{x}_\alpha(t_1)$ and $\til{x}_\gamma(t_2)$ is dealt with exactly like in \ref{provaBorel}. The relations \bref{parabWU} and \bref{parabU} for $\gamma \neq \eta$ belong to $\ker(f)$ by Lemma \ref{DERTrick}. Finally, the additive relations \bref{parabADD} belong to $\ker(f)$ because $\gera{\mf{X}_\gamma, \mc{H}_\gamma} \cong \mcB_2^0(R)$ surjects onto $\gera{\til{\mf{X}}_\gamma,\mc{H}_\gamma} \leq \til{\P}_I(R)$ via $f$. Therefore $\mc{S}_I \subseteq \ker(f)$.

\subsubsection*{Case 2 -- \texorpdfstring{$R$}{R} is not very bad} \label{casoNVB} This time, we add no further relations besides the ``obvious'' ones already given to obtain a finite presentation of $\P_I(R)$. The standing assumptions that the {\bf NVB} condition allows us to make is that \emph{the structure constants of the commutator formulae are all invertible and we assume them to be in the generating set} $A^{[c_\Phi]}$ of the group of units $R^\times$. In particular, $(C^{\gamma, \eta}_{m,n})^{\pm 1}\cdot x_i \in \til{T}$ for every $x_i \in T_0$. Now, let $\til{\mc{S}}_I$ be the set of all relations \bref{tilParabWT}, \bref{tilParabWU} and \bref{tilParabU} from the previous section together with the following relations regarding the structure constants.
\begin{equation} \label{StructureConstantsRel}
\til{x}_\delta((C_{m,n}^{\gamma,\eta})^{-1} \cdot x_i)^{C_{m,n}^{\gamma,\eta}} = \til{x}_\delta(1 \cdot x_i),
\end{equation}
\begin{equation} \label{StructureConstantsRel2}
h_\alpha(C_{m,n}^{\gamma,\eta}) \til{x}_\delta(x_i) h_\alpha(C_{m,n}^{\gamma,\eta})^{-1} = \til{x}_\delta(x_i) \mbox{ whenever } (\delta, \alpha) = 1.
\end{equation}

 We shall prove that the finitely presented group
\[
\til{\P}_I(R) = \gera{\mc{X} \cup \til{\mc{Y}} \tq \mc{R} \cup \til{\mc{S}}_I}
\]
is isomorphic to $\P_I(R)$.  The set-up is the same as in the previous section, the goal being to show that the relations $\mc{S}_I$ live in $\ker(f)$. Following the previous cases, most of the relations in $\mc{S}_I$ were already dealt with. For the commutator relations, it suffices to prove that, for all $t, s \in R$ and $\gamma \in \Phi^+ \backslash \Phi_{\Ext(I)}$,
\[
f([x_\gamma(t), x_\gamma(s)]) = 1,
\]
since the remaining commutator relations will then follow from Lemma \ref{DERTrick}, as done in the previous cases. For this purpose, we first redefine unipotent root subgroups, now in the finitely presented group $\til{\P}_I(R)$, and recover the analogue of Lemma \ref{derTrick}.

We remark that, since the torus $\H$ is contained in the extended Levi factor $\LE_I(R)$, whose presentation is included in that of $\til{\P}_I(R)$, we freely can (and do) make full use in $\til{\P}_I(R)$ of relations between the semi-simple root elements and simplify expressions. In particular, if $u = v_1^{n_1} \cdots v_\xi^{n_\xi} \in R^\times = \gera{A^{[c_\Phi]}}$, we write $h_\delta(u) = h_\delta(v_1)^{n_1} \cdots h_\delta(v_\xi)^{n_\xi}$. Given $\gamma \in \Phi^+ \backslash \Phi_{\Ext(I)}$, we let
\[
\til{\mf{X}}_\gamma = \gera{\set{h \til{x}_\gamma(t) h^{-1} \tq t \in \til{T},~ h \in \H}} \leq \til{\P}_I(R).
\]
The next lemma shows, in particular, that $\til{\mf{X}}_\gamma = \gera{\set{h_\gamma \til{x}_\gamma(t) h_\gamma^{-1} \tq t \in \til{T},~ h_\gamma \in \H_\gamma}}$.
\begin{lem} \label{2.3'}
 Given two distinct roots $\alpha, \beta$ with $\beta \neq -\alpha$, there exist a subtorus $H_{\alpha,\beta}(R) \leq \H$ and a non-zero integer $m_{\alpha,\beta}$ such that $H_{\alpha,\beta}(R)$ centralizes $\mf{X}_\alpha$ and
 \[
  h(u) \til{x}_\beta (v x_i) h(u)^{-1} = h_\beta (u)^{m'} \til{x}_\beta (v' x_i) h_\beta (u)^{-m'},
 \]
where $m' = m'(m_{\alpha,\beta},v) \in \Z$ and $v' = v'(m_{\alpha, \beta}, v) \in A^{[c_\Phi]}$ are unique such that $u^{m_{\alpha,\beta}} v = u^{2m'} v'$.
\end{lem}
\begin{proof}
 As in the proof of Lemma \ref{derTrick}, take $p, q \in \Z$ with $2p - q \cdot (\alpha, \beta) = 0$ and set $h(u) := h_\beta(u)^{-q} h_\beta(u)^p \in \til{\P}_I(R),~ H_{\alpha,\beta}(R) := \gera{\set{h(u) \tq u \in R^\times}} \leq \til{\P}_I(R)$ and $m_{\alpha, \beta} := p \cdot (\beta, \alpha) - 2q \neq 0$. The equation stated follows from iterated applications of \bref{tilParabWT}. As for the first claim, by induction and \bref{tilParabWT} there exist unique $v' \in A^{[c_\Phi]}$ and $n' \in \Z$ such that $u^{-q(\alpha,\beta)}v = u^{2n'}v'$, thus
 \[
  h(u) \til{x}_\alpha(vx_i) h(u)^{-1} = h_\alpha(u)^p h_\beta(u)^{-q} \til{x}_\alpha(vx_i) h_\beta(u)^q h_\alpha(u)^{-p}
  \]
  \[
  = h_\alpha(u)^{p+n'} \til{x}_\alpha(v'x_i) h_\alpha(u)^{-p-n'}.
 \]
On the other hand, by our choice of $p$ and $q$, we have that $2p = q(\alpha, \beta) = -2n'$. Therefore $h_\alpha(u)^{p+n'} = h_\alpha(u)^{p-p} = 1$ and the lemma follows.
\end{proof}

In the sequel we simplify the proofs by making use of more explicit Chevalley commutator formulae, though with no loss of generality since the proofs are analogous if the ordering of the roots (and thus the formulae) change. We refer the reader to \cite[Chapter 10]{Steinberg} for explicit formulae and structure constants in types $B$ and $G$. We shall often simplify the notation on the structure constants, writing e.g. $B, C, D, E,...$ instead of $C_{m,n}^{\gamma,\eta}, C_{a,b}^{\alpha,\beta},...$ and so on.

Let $r, s \in R$ and $\gamma \in \Phi^+ \backslash \Phi_{\Ext(I)}$. We want to show that 
\begin{align} \label{asterisco}
\begin{split}
 1 = f([x_\gamma(r), x_\gamma (s)]) = & \displaystyle \left[ \prod_{l = 0}^{\nu} h_\gamma(w_l)^{k_l} \til{x}_\gamma(u_l x_l)^{a_l} h_\gamma(w_l)^{-k_l} \right., \\
 & \displaystyle \left. \prod_{m = 0}^{\nu} h_\gamma(z_m)^{e_m} \til{x}_\gamma(v_m x_m)^{b_m} h_\gamma(z_m)^{-e_m} \right],
\end{split}
 \tag{**}
\end{align}
where
\[
r = \sum_{l = 0}^{\nu}a_l w_l^{2k_l} u_l x_l \mbox{ with } a_l \in \Z, w_l \in \gera{A^{[c_\Phi]}}, u_l x_l \in \til{T}
\]
and
\[
 s = \sum_{m = 0}^{\nu} b_m z_m^{2e_m} v_m x_m, \mbox{ with } b_m \in \Z, z_m \in \gera{A^{[c_\Phi]}}, v_m x_m \in \til{T}.
\]
Relations \bref{tilParabU} already give us $[\til{x}_\gamma(t), \til{x}_\gamma(s)] = 1$ for all $t, s \in \til{T}$. Since $\gamma \in \Phi^+ \backslash \Phi_{\Ext(I)}$, we may choose $\alpha \in \Phi_I \cup \Phi^+$ and $\beta \in \Phi^+$ such that $\alpha + \beta = \gamma$, which exist either by Lemma \ref{Humphreys} or by Lemma \ref{adjacencia}. Now, for each $s \in \til{T}$, using the commutator relations \bref{tilParabWU} and \bref{tilParabU} at our disposal, together with \bref{StructureConstantsRel}, we obtain from the explicit commutator formulae the following equations (which not necessarily cover all possibilities!) for $\til{x}_\gamma(s)$, depending on the type of the subsystem $\Phi_{\set{\alpha,\beta}}$.

\begin{equation} \label{possibilidades}
 \til{x}_\gamma(s) =
 \begin{cases}
  [\til{x}_\alpha(s), \til{x}_\beta(C^{-1})]^{\pm 1}, \mbox{ if } (m\alpha + n \beta \in \Phi_{\set{\alpha,\beta}} \iff m=n=1); \\
  [\til{x}_\alpha(s), \til{x}_\beta(1)] \til{x}_{\alpha + 2\beta}(s)^{\mp 1}, \mbox{ if } \gamma, \gamma + \beta \in \Phi_{\set{\alpha,\beta}} = B_2; \\
  [\til{x}_\alpha(s), \til{x}_\beta(D^{-1})] \til{x}_{\alpha + \gamma}(s^2 D^{-1})^{-3} \til{x}_{\gamma + \beta}(s D^{-2})^{3}, \mbox{ if } \gamma \mbox{ is short and} \\
	\Phi_{\set{\alpha,\beta}} \mbox{ is of type } G_2,
 \end{cases}
\end{equation}
where $C, D$ are shortenings for the appropriate structure constants involved in each type. (Warning: Here we are slightly misusing notation. In fact, the powers of $s$ or of $D^{-1}$ and their products are not necessarily allowed in $\til{x}_\gamma(\cdot)$, for they need not be elements of $\til{T}$. We should instead write $h\til{x}_\gamma(v' y') h^{-1}$ for some $v' y' \in \til{T}$ and some $h \in \H$ in each misused occurrence in the above. However, since we are dealing in the sequel with commutators of products of the form $h\til{x}_\gamma(v' y') h^{-1}$, this abuse of notation does not affect our arguments.)

Applying \bref{possibilidades} to \bref{asterisco} only in the expression of $f(x_\gamma(s))$, we conclude that \bref{asterisco} holds once we prove the following. For all $h, g \in \H_\gamma$ and $t, s \in \til{T}$,
\begin{equation} \label{contafinal}
 [h \til{x}_\gamma(s) h^{-1}, g \til{x} g^{-1}] = 1,
\end{equation}
where $\til{x}$ is a product of the form of (one of the cases of) the right hand side of \bref{possibilidades}. For the proof we shall need the following.
\begin{obs} \label{cortacaminho}
For all $\delta_1, \delta_2 \in \Phi$ with $\delta_1 + \delta_2$ being the only linear combination of those roots which lies in $\Phi$, one has
\[
\displaystyle \left[ \prod_{i=1}^n h_i \til{x}_{\delta_1}(t_i) h_i^{-1}, \prod_{j=1}^m h_j \til{x}_{\delta_2}(t_j) h_j^{-1} \right] = 1
\]
for all $h_i, h_j \in \H$ and $t_i, t_j \in \til{T}$.
\end{obs}
To see this, just apply the proof of Steps 0, 1 and 2 of Lemma \ref{DERTrick} to the set-up above for $\til{\P}_I(R)$, replacing the use of Lemma \ref{derTrick} there by Lemma \ref{2.3'}---observe that Steps 0, 1 and 2 do not depend on the commutativity relations \bref{faltouabeliano} and thus can be carried over almost verbatim to the present context. We now prove \bref{contafinal} based on the three cases of \bref{possibilidades}.

{\bf Case 1} Suppose $\gamma = \alpha + \beta$ is the only linear combination of $\alpha$ and $\beta$ in $\Phi_{\set{\alpha,\beta}}$. We have
\[
f([h x_\gamma(t) h^{-1}, g x_\gamma(s) g^{-1}]) = [h \til{x}_\gamma(t) h^{-1}, g \til{x}_\gamma(s) g^{-1}] 
\]
\[
= [h \til{x}_\gamma(t) h^{-1}, g [\til{x}_\alpha(s), \til{x}_\beta(C^{-1})]^{\pm 1} g^{-1}] = 1,
\]
by Remark \ref{cortacaminho}.

{\bf Case 2} Suppose $\Phi_{\set{\alpha,\beta}}$ is of type $B_2$ with $\gamma, \gamma + \beta \in \Phi_{\set{\alpha,\beta}}$. By \bref{possibilidades} and Remark \ref{cortacaminho}, one has
\[
h \til{x}_\gamma(t) h^{-1} g \til{x}_\gamma(s) g^{-1} = 
\]
\[
= h \til{x}_\gamma(t) h^{-1} g ( \til{x}_\alpha(s)^{\pm 1} \til{x}_\beta(1)^{\pm 1} \til{x}_\alpha(s)^{\mp 1} \til{x}_\beta(1)^{\mp 1} \til{x}_{\gamma + \beta}(s)^{\mp 1} ) g^{-1}
\]
\begin{equation} \label{eq1doB2}
= g \til{x}_\alpha(s)^{\pm 1} g^{-1} h \til{x}_\gamma(t) h^{-1} g (\til{x}_\beta(1)^{\pm 1} \til{x}_\alpha(s)^{\mp 1} \til{x}_\beta(1)^{\mp 1} \til{x}_{\gamma + \beta}(s)^{\mp 1}) g^{-1}. \tag{$\diamond$}
\end{equation}
By Lemma \ref{2.3'}, we can find $h_1, g_1 \in \H$ and $t', a \in \til{T}$ such that
\begin{align} \label{eq2doB2}
\begin{split}
h \til{x}_\gamma(t) h^{-1} = h_1 \til{x}_\gamma(t') h_1^{-1},~ & g \til{x}_\beta(1) g^{-1} = g_1 \til{x}_\beta(a) g_1^{-1}, \\
h_1 \til{x}_\beta(a) = \til{x}_\beta(a) h_1 \mbox{ and } & g_1 \til{x}_\gamma(t') = \til{x}_\gamma(t') g_1.
\end{split} \tag{$\star$}
\end{align}
Applying \bref{eq2doB2} and Relations \bref{tilParabWU}, \bref{tilParabU} to \bref{eq1doB2}, we obtain
\begin{align*}
(\diamond) = & g \til{x}_\alpha(s)^{\pm 1} g^{-1} g \til{x}_\beta(1)^{\pm 1} g^{-1} h_1 g_1 \til{x}_{\gamma + \beta}(t'a)^{\pm 2} g_1^{-1} h_1^{-1} h \til{x}_\gamma(t) h^{-1} \times \\
& \times g (\til{x}_\alpha(s)^{\mp 1} \til{x}_\beta(1)^{\mp 1} \til{x}_{\gamma + \beta}(s)^{\mp 1}) g^{-1} \\
 = & g (\til{x}_\alpha(s)^{\pm 1} \til{x}_\beta(1)^{\pm 1} ) g^{-1} h_1 g_1 \til{x}_{\gamma + \beta}(t'a)^{\pm 2} g_1^{-1} h_1^{-1} g \til{x}_\alpha(s)^{\mp 1} g^{-1} \times \\
& \times g \til{x}_\beta(1)^{\mp 1} g^{-1} h_1 g_1 \til{x}_{\gamma + \beta}(t'a)^{\mp 2} g_1^{-1} h_1^{-1} h \til{x}_\gamma(t) h^{-1} g \til{x}_{\gamma + \beta}(s)^{\mp 1} g^{-1} \\
= & g (\til{x}_\alpha(s)^{\pm 1} \til{x}_\beta(1)^{\pm 1} \til{x}_\alpha(s)^{\mp 1} \til{x}_\beta(1)^{\mp 1} \til{x}_{\gamma + \beta}(s)^{\mp 1}) g^{-1} h \til{x}_\gamma(t) h^{-1} \\
= & g \til{x}_\gamma(s) g^{-1} h \til{x}_\gamma(t) h^{-1},
\end{align*}
because the $h' \til{x}_{\gamma + \beta}(r) h'^{-1}$ commute with the other terms above by Remark \ref{cortacaminho} and the previous case. If, on the other hand, $\gamma + \beta \notin \Phi_{\set{\alpha, \beta}}$, then we are back in the situation of Case 1, so Case 2 is concluded.

{\bf Case 3} Assume $\Phi_{\set{\alpha,\beta}}$ to be of type $G_2$. If $\gamma$ is long we may take $\alpha, \beta$ as in the situation of Case 1. Otherwise, and since the case $\P_{\delta}(R) \leq E_{G_2}^{sc}(R),~ \delta$ long, is excluded, we may take $\alpha, \beta$ such that the third equality of \bref{possibilidades} apply. We thus have
\begin{align}
\begin{split}
h \til{x}_\gamma(t) h^{-1} g \til{x}_\gamma(s) g^{-1} = & h \til{x}_\gamma(t) h^{-1} \times \\
 \times & g ( [\til{x}_\alpha(s), \til{x}_\beta(D^{-1})] \til{x}_{\alpha + \gamma}(s^2 D^{-1})^{-3} \til{x}_{\gamma + \beta}(s D^{-2})^{3} ) g^{-1}.
\end{split} \tag{$\bigtriangleup$}
\end{align}
The rest of the proof is entirely analogous to the previous case. Indeed, pick $h_1, g_1, h_2, g_2 \in \H$ and $t', s', t'', s'' \in \til{T}$ such that
\begin{align} \label{eq1doG2}
\begin{split}
h \til{x}_\gamma(t) h^{-1} = h_1 \til{x}_\gamma(t') h_1^{-1},~ & g \til{x}_\alpha(s) g^{-1} = g_1 \til{x}_\alpha(s') g_1^{-1}, \\
h_1 \til{x}_\alpha(s') = \til{x}_\alpha(s') h_1,~ & g_1 \til{x}_\gamma(t') = \til{x}_\gamma(t') g_1 \\
h \til{x}_\gamma(t) h^{-1} = h_2 \til{x}_\gamma(t'') h_2^{-1},~ & g \til{x}_\beta(D^{-1}) g^{-1} = g_2 \til{x}_\beta(s'') g_2^{-1}, \\
h_2 \til{x}_\beta(s'') = \til{x}_\beta(s'') h_2 \mbox{ and } & g_2 \til{x}_\gamma(t'') = \til{x}_\gamma(t'') g_2,
\end{split} \tag{$\dagger$}
\end{align}
which exist by Lemma \ref{2.3'}. Since $\gamma + \gamma + \alpha,~ \gamma + \gamma + \beta \notin \Phi_{\set{\alpha, \beta}}$, we obtain from Equalities \bref{eq1doG2}, Relations \bref{tilParabWU}, \bref{tilParabU} and Remark \ref{cortacaminho} that
\begin{align*}
(\bigtriangleup) = & h_1 g_1 \til{x}_{\gamma + \alpha}(t's')^{-3} g_1^{-1} h_1^{-1} g \til{x}_{\alpha}(s) g^{-1} h_2 g_2 \til{x}_{\gamma + \beta}(t''s'')^{-3} g_2^{-1} h_2^{-1} \times \\
\times & g \til{x}_\beta(D^{-1}) g^{-1} h_1 g_1 \til{x}_{\gamma + \alpha}(t's')^{3} g_1^{-1} h_1^{-1} g \til{x}_{\alpha}(s)^{-1} g^{-1} \times \\
\times & h_2 g_2 \til{x}_{\gamma + \beta}(t''s'')^{3} g_2^{-1} h_2^{-1} g \til{x}_\beta(D^{-1})^{-1} g^{-1} g \til{x}_{\gamma + \alpha}(s^2D^{-1})^{-3} g^{-1} \times \\
\times & g \til{x}_{\gamma+\beta}(sD^{-2})^{3} g^{-1} h \til{x}_\gamma(t) h^{-1} \\
= & g \til{x}_\gamma(s) g^{-1} h \til{x}_\gamma(t) h^{-1},
\end{align*}
as required.

From the three cases above, we conclude that all relations \bref{parabU} are in $\ker(f)$.

It remains to show that the additive relations \bref{parabADD} also lie in $\ker(f)$. It suffices to prove this for the simple roots $\beta$ in $\Phi^+ \backslash \Phi_{\Ext(I)}$. Such a simple root is necessarily adjacent to an element of $I$, so by Lemma \ref{adjacencia} we may assume that the root subgroup $\til{\mf{X}}_\beta$ lies in the unipotent radical of a parabolic subgroup in type $A_2$, $B_2$ or $G_2$, whose Levi factor is generated by $\mf{X}_\alpha, \mf{X}_{-\alpha}$ and $\mc{H}_\beta$ for some $\alpha \in I$. Moreover, $\set{\alpha, \beta}$ is a basis for the underlying root subsystem.

From now on, we make full use of the commutator relations without further references, for they were already obtained from the previous steps of the proof together with Lemma \ref{DERTrick}. The point now is that the additive relations hold in $\gera{\mf{X}_\alpha, \mf{X}_{-\alpha}, \mc{H}_\beta}$ because they do in $\LE_I(R)$. Consequently, they also hold in $\mf{X}_\delta$, where $\delta \in \Phi_{\set{\alpha,\beta}}$ is a root of maximal height. Let $a \in \mc{A}$ be a defining additive relation. We want to show that $\til{x}_\beta(a) = 1$. If $\Phi_{\set{\alpha,\beta}}$ is of type $A_2$ or if $\alpha$ is a short in $A_2,~ B_2$ or $G_2$ then, since $R$ is not very bad for $\Phi$, we can write
\[
\til{x}_\beta(a) = [\til{x}_{\pm \alpha}(a), \til{x}_\eta(C)]
\]
for some $\eta \in \Phi_{\set{\alpha,\beta}}$, where $C$ is (an inverse of) a structure constant (possibly $1$). This implies that $\til{x}_\beta(a) = 1$, since $\til{x}_{\pm \alpha}(a) = 1$.

Suppose $\beta$ is short and $\Phi_{\set{\alpha,\beta}} = G_2$. This is actually the case of the excluded parabolic subgroup $\P_{\set{\alpha}}(R) \leq E_{G_2}^{sc}(R),~ \alpha$ long (recall that we are assuming the commutator relations hold). We include it here because the arguments also illustrate the proof for the remaining case, type $B_2,~ \beta$ short. We have then the following equations.
\[
\til{x}_{\alpha+\beta}(a) \til{x}_{\alpha + 2\beta} (a)^{-1} = 1;
\]
\[
\til{x}_\beta(a)^{\pm B} \til{x}_{\alpha +2\beta}(a)^{\pm C}.
\]
Since $\til{x}_\eta(a)$ vanishes for $\eta$ of large height, the first equation implies
\[
\til{x}_{\alpha + 2\beta}(a)^{\pm D} = 1.
\]
Here, $B, C, D$ are shortenings for structure constants. Observe that $\alpha + 2\beta$ belongs to a root subsystem of type $A_2$ so that we can find a root $\gamma$ such that $(\alpha + 2\beta, \gamma) = 1$. Since the structure constants are invertible, there is a $u \in A^{[c_\Phi]}$ such that
\[
h_\gamma(u) \til{x}_{\alpha + 2\beta}(a)^{\pm 1} h_\gamma(u)^{-1} = \til{x}_{\alpha + 2\beta}(a)^{\pm D} = 1.
\]
Hence $\til{x}_{\alpha + 2\beta}(a)$, and so $\til{x}_\beta(a)^{\pm B}$, vanish. Since $\beta$ also lies in a root subsystem of type $A_2$, it follows that $\til{x}_\beta(a) = 1$. 

If $\Phi_{\set{\alpha,\beta}} = B_2$, an entirely analogous argument shows that $\til{x}_\beta(a) = 1$ in that case too. Therefore $\mc{S}_I \subseteq \ker(f)$, which concludes the proof of the theorem.

\section{Proof of Theorem \texorpdfstring{\ref{B}}{C}} \label{aplicacoes}

\noindent

We begin by reducing the problem to the case of Chevalley--Demazure groups. Let $\mbf{G}$ be a split, connected, reductive, linear algebraic group defined over a global field $\K$, let $\OS \subset \K$ be a ring of $S$-integers---i.e. a Dedekind domain of arithmetic type---and let $\mbf{P}$ be a proper parabolic subgroup. Since every parabolic in $\mbf{G}$ is conjugate to a standard one, we may restrict ourselves to the standard parabolic subgroups with respect to an arbitrary, but fixed, maximal split torus. The group $\mbf{G}$ fits into the following diagram.
\[
\xymatrix{
& & \G \ar@{->>}[d]^{f_1} \\
\mbf{G} & \ar@{->>}[l]^{f_2} \mc{R}\mbf{G} \times \mbf{G}' \ar@{->>}[r] & \mbf{G}'
}
\]
Here $\mbf{G}'$ is semi-simple, the maps $f_1$ and $f_2$ are central isogenies, $\mc{R}\mbf{G}$ is the radical of $\mbf{G}$ and the upper group $\G$ is simply connected. By taking the corresponding diagram restricted to parabolic subgroups, it follows that $S$-arithmetic subgroups of $\mbf{P} \leq \mbf{G}$ are finitely generated (resp. finitely presented) if and only if so are the $S$-arithmetic subgroups of the corresponding parabolic $\P \leq \G$ (see e.g. \cite[91]{Behr98}). Since $\G$ is split, connected, simply connected and semi-simple, it is in fact a Chevalley--Demazure group scheme of simply connected type, so the group $\G(\OS)$ of $\OS$-points is well-defined and is in fact $S$-arithmetic since it fixes the $\OS$-lattice acted upon by $\G(\K)$.
As $S$-arithmetic subgroups of a given algebraic group are commensurable, we may restrict ourselves to the group $\P(\OS) \leq \G(\OS)$. Finally, the condition $|S| > 1$, if $\K$ is a global function field, guarantees that the parabolic subgroups of $\P(\OS)$ are finitely generated by O'Meara's structure theorem \cite[Thm. 23.2]{O'Meara}.

Let $\Gamma \leq \mbf{P}$ be $S$-arithmetic. We now break the proof in three, according to the given hypotheses of Theorem \ref{B}.

\subsection{Part \texorpdfstring{(\ref{B.1})}{C.i}} If $\carac(\K) = 0$, Abels' theorem \cite{Abels} implies that $\B_0^2(\OS)$ is always finitely presented, so $\OS$ satisfies the {\bf QG} condition. Now, the extended Levi factor is an extension, by a torus, of a direct product of Borel subgroups or reductive groups (cf. Section \ref{Levi}). From Abels' theorem and \cite[Thm. 6.2]{BoSe}, the extended Levi factors in characteristic zero are always finitely presented, so Part \bref{B.1} follows from Theorem \ref{A}.

\subsection{Part \texorpdfstring{(\ref{B.2a})}{C.iia}} The following remark gives meaning to the statement of Theorem \ref{B}, Part \bref{B.2a}.

\begin{lem}
Let $r : \mbf{H} \onto \mbf{B}$ be a $k$-retract of $k$-split, connected, linear algebraic groups. If $\Lambda \leq \mbf{B}$ is $S$-arithmetic, then $\Lambda$ is finitely presented (resp. finitely generated) whenever an $S$-arithmetic subgroup $\Delta \leq \mbf{H}$ is finitely presented (resp. finitely generated).
\end{lem}
\begin{proof}
Since $r : \mbf{H} \onto \mbf{B}$ is a $k$-retract, we may (and do) identify $\mbf{B}$ with a $k$-closed subgroup of $\mbf{H}$. Without loss of generality, fix a $k$-embedding $\mbf{H} \into \GL_n$, for some $n$, so that both $\mbf{H}$ and its subgroup $\mbf{B}$ are seen as $k$-closed subgroups of the same $\GL_n$. Let $\Lambda$ be the $S$-arithmetic subgroup $\mbf{B} \cap \GL_n(\OS)$ of $\mbf{B}$, and set $\Delta := r^{-1}(\Lambda)$. Then $r$ induces a group retract $r|_{\Delta} : \Delta \onto \Lambda$. Now,
\begin{align*}
\Delta = & r^{-1}(\mbf{B} \cap \GL_n(\OS)) = \{ g \in \mbf{H} \tq r(g) \in \mbf{B} \mbox{ and } r(g) \mbox{ has entries in } \OS \} \\ = & \mbf{H} \cap \GL_n(\OS),
\end{align*}
so $\Delta$ is also $S$-arithmetic. Since $S$-arithmetic subgroups of a linear algebraic group are commensurable, the claim follows from Lemma \ref{Stallings} together with the retract $r|_\Delta$.
\end{proof}

We first prove the necessity part of \bref{B.2a}. From now on we assume $\carac(\K) > 0$. Under the hypothesis of \bref{B.2a}, we have that the soluble group $\Addi(\OS) \rtimes \mbf{T}(\OS)$ is finitely presented. We claim that $\B_2^0(\OS)$ is also finitely presented, so that $\OS$ is \QGff. 

To see this, consider $\Addi \rtimes \mbf{T}$ as a subgroup of a Borel subgroup in a (universal) Chevalley--Demazure group, with $\Addi$ being the unipotent root subgroup associated to the first simple root. Viewing this as the base root as in \cite[p. 625]{Bux04}, the proofs of the lemmata given in \cite[Section 5]{Bux04} carry over to our case and it follows from K. Brown's criterion \cite{Brown} and \cite[Lemma 5.7]{Bux04} that $\Addi(\OS) \rtimes \mbf{T}(\OS)$---whence $\Gamma$---is finitely presented only if $|S| \geq 3$.

Let us now point out an alternative proof to the fact that $\B_2^0(\OS)$ is finitely presented which does not depend on the results of \cite{Bux04}. For this we need the following observation. If the group of units $\Mult(R)$ is finitely generated, then the finite presentability of the groups
\[
B_1(R) := \begin{pmatrix} * & * \\ 0 & 1 \end{pmatrix} \leq \GL_2(R) \mbox{ and } \Bzero = \begin{pmatrix} * & * \\ 0 & * \end{pmatrix} \leq \SL_2(R)
\]
is equivalent. In effect, suppose $\Bzero$ is finitely presented. Then $\Addi(R)$ is a tame $\Z[\Mult(R)]$-module (see \cite{BieriStrebel}) with action given as in the beginning of Section \ref{teoremao}. Now look at the induced action from the group of squares $(\Mult(R))^2 := \set{g^2 \tq g \in \Mult(R)}$ given by
\[
 \xymatrix@R=2mm{
 {(\Mult(R))^2 \times \Addi(R)} \ar[r]
 & {\Addi(R)} \\
 {(u^2, r) } \ar@{|->}[r]
 & {u^2 r.}
 }
\]
Since every character $v: (R^\times)^n \to \R$ extends to a character $\bar{v} : R^\times \to \R$, it follows that $\Addi(R)$ is a tame $\Z[(\Mult(R))^2]$-module, so $\Addi(R) \rtimes (\Mult(R))^2$ is finitely presented. Since $\Addi(R) \rtimes (\Mult(R))^n$ has finite index in $\Addi(R) \rtimes \Mult(R) \cong B_1(R)$, it follows that $B_1(R)$ itself is finitely presented. The converse is clear.

Adapting the retraction arguments from \cite[Section 4]{Bux04}, we prove in \cite{euAbels} using fairly simple arguments that, in particular, the finiteness length of $\B_2^0(\OS)$ bounds that of $\Addi(\OS) \rtimes \mbf{T}(\OS)$ from above. We thus have the following.
\begin{pps} \label{eu}
Let $\mbf{B} = \Addi \rtimes \mbf{T}$ be a $\K$-split, connected, (non-nilpotent) soluble linear algebraic group, and let $\Lambda \leq \mbf{B}$ be $S$-arithmetic. If $\Lambda$ is finitely generated (resp. finitely presented), then so is $\B_2^0(\OS)$.
\end{pps}
By the remark above for $B_1(\OS)$, by Proposition \ref{eu} and by \cite{Bux0}, it follows that $\Gamma$ is finitely presented only if $|S| \geq 3$.

Conversely, if $|S| \geq 3$ then the $S$-arithmetic Borel subgroups are finitely presented, either by the remark above for $B_1(\OS)$ and by the arguments from Section \ref{provaBorel} for Theorem \ref{A} in case $I = \vazio$, or by \cite{Bux04}. Furthermore, the reductive part of an extended Levi factor will always be finitely presented as well, by Behr's rank theorem \cite{Behr98}. Thus, $\Gamma$ is finitely presented by Theorem \ref{A}.

\subsection{Part \texorpdfstring{(\ref{B.2b})}{C.iib}} Assume $\carac(\K) > 0$ and that $\K$ is not very bad for the underlying root system of $\G$. Then the subring $\OS$ is also {\NVB} since it contains the (finite) prime field by \cite[23.1 and 23.2]{O'Meara}. Recall that $\P = \U \rtimes \L \leq \G$ is a standard parabolic subgroup of the simply connected, semi-simple group $\G$, where $\U$ is the unipotent radical and $\L$ the Levi factor. We then have $\L = (\prod_i \G_{\Phi_i}) \rtimes \H$, where $\H$ is a torus and each $\G_{\Phi_i}$ is a Chevalley--Demazure group scheme (cf. Section \ref{parabolicos}). Furthermore, $\U(\OS)$ admits a presentation as given in Lemma \ref{presK}. 

Suppose $\rk(\G_{\Phi_i}) \geq 2$ for all $i$. In this case, we know from \cite[Cor. 4.6]{Matsumoto} and \cite[Thm. 14.1]{BassMilnorSerre} that each $\G_{\Phi_i}(\OS)$ equals its elementary subgroup $E_{\Phi_i}(\OS)$, so $\P(\OS)$ has the form given in Section \ref{parabolicos} and thus Theorem \ref{A} applies directly, and we are done.

Assume then that $\rk(\G_{\Phi_i}) = 1$, possibly the exceptional case from Theorem \ref{A} where $\P(\OS) = \P_{\set{\alpha}}(\OS) \leq \G_{G_2}(\OS)$, $\alpha$ long. Here, $\P(\OS)$ is finitely presented only if so is $\G_{\Phi_i}(\OS)$, which in turn is finitely presented if and only if so is $\G_{\Phi_i}'(\OS)$ by the diagram from the beginning of Section \ref{aplicacoes}. But the latter group is isomorphic to $\SL_2(\OS)$, which is finitely presented only if $|S| \geq 3$ by Behr's theorem \cite{Behr98}. In this case, the Borel subgroup $\B_2^0(\OS)$ is also finitely presented as seen above, so the arguments from Section \ref{provaParab1} apply and thus Theorem \ref{A} holds in this case, too. This concludes the proof of Theorem \ref{B}.

\section{Concluding remarks} \label{remarks}

\subsection{} The reader might have noticed that we ``forgot'' the case of the general linear group along the text. Considering the general elementary group $\mathrm{GE}_n \leq \GL_n$ (see e.g. \cite{Silvester}), we may define the extended Levi factor for $\mathrm{GE}_n$ just as in Section \ref{exemplao} or as in Definition \ref{defLE}, replacing the torus $\H$ by the subgroup $D_n \leq \mathrm{GE}_n$ of diagonal matrices. Since $D_n$ contains the standard torus of $\SL_n$, the same methods from Section \ref{teoremao} apply and we obtain analogous results---in fact, $\mathrm{GE}_n$ fits the characterization given in Corollary \ref{classifsimplylaced}. The extra generators and relations that will occur pose no extra problems---they are analogous to those of Steinberg and are well-known \cite{Silvester}.

\subsection{} As seen in Sections \ref{provaBorel} and \ref{provaParab1}, the assumption that $\Bzero$ be finitely presented allows for a result with no restrictions on the root system nor on the characteristic of the base ring. In particular, the exceptional case $\P_{\set{\alpha}}(R) \leq E_{G_2}^{sc}(R)$, $\alpha$ long, only shows up if $\Bzero$ is not finitely presented. Theorem \ref{A} could be restated as follows.

\setcounter{thml}{0}
\begin{thml}[rewritten]
Let $\ueCD(R)$ be a universal elementary Chevalley--Demazure group for which its (standard) parabolic subgroups are finitely generated.
\begin{itemize}
\item If $\Bzero$ is finitely presented, then a standard parabolic subgroup $\P_I(R) \leq \ueCD(R)$ is finitely presented if and only if so are its extended Levi factors;
\item Otherwise, and if $R$ is \NVBff, then $\P_I(R)$ is finitely presented if and only if so are its extended Levi factors, except possibly in the case where $I = \set{\alpha}$ with $\alpha$ a long root in the root system of type $G_2$.
\end{itemize}
\end{thml}

Yet another formulation of Theorem \ref{A} can be given by reinterpreting its proof. Recall that a group $G$ is finitely presented with respect to a subgroup $H \leq G$ if there exist finite subsets $\mc{X} \subseteq G$ and $\mc{R} \subseteq H \ast F_{\mc{X}}$ such that $G \cong H \ast F_{\mc{X}} / \langle \langle \mc{R} \rangle \rangle$. The constructions from Section \ref{teoremao} imply, under the assumptions of Theorem \ref{A} above, that $\mc{P}_I(R)$ is always finitely presented with respect to any of its extended Levi factors. Hence, a parabolic $\mc{P}_I(R)$ is finitely presented if and only if so are its extended Levi factors.

\subsection{} \label{precisao} We observe that the only instance of the proof of Theorem \ref{A} where we had to exclude the parabolic $\P_{\set{\alpha}}(R) \leq E_{G_2}^{sc}(R)$, $\alpha$ long, was when dealing with commutativity relations for unipotent root elements in Section \ref{casoNVB}. In the exceptional case, applying the same procedure as in \ref{casoNVB} to the (short) root $\alpha + \beta$---here, $\Phi_{\set{\alpha, \beta}} = G_2$---leads to equations, for instance, of the form
\begin{align*}
h \til{x}_{\alpha+\beta}(t) h^{-1} g \til{x}_{\alpha + \beta}(s) g^{-1} = & g \til{x}_{\alpha + \beta}(s) g^{-1} h \til{x}_{\alpha+\beta}(t) h^{-1} h_1 \til{x}_{2\alpha+3\beta}(a)^{-9} h_1^{-1} \times \\
\times & h_2 \til{x}_{2\alpha+3\beta}(b)^{3} h_2^{-1} h_3 \til{x}_{\alpha+3\beta}(c)^6 h_3^{-1},
\end{align*}
where $h_i, h, g \in \H$. It is likely that the $2\alpha + 3\beta$ terms cancel, but it is not clear how all extra terms should vanish. We do not discard the possibility that the same methods used here still apply for this exceptional case, though we could not find an alternative route to verify this. As a test case, one could consider the following.
\begin{pbl}
Prove that the parabolic subgroup $\P_{\set{\alpha}}(\F_5[t,t^{-1}]) \leq E_{G_2}(\F_5[t,t^{-1}])$, with $\alpha$ a long root in type $G_2$, is finitely presented if and only if its Levi factor is finitely presented (even though none of them admits a finite presentation). 
\end{pbl}

\subsection{} Theorem \ref{A} might be strengthened by proving that the finite presentability of the Borel subgroup $\Bzero$ implies that of (any) universal elementary Chevalley--Demazure group $\ueCD(R)$. This holds for Dedekind rings of arithmetic type by Borel--Serre's and Behr's theorems, and was often used in the proof of Theorem \ref{B}. Whether this is true in general is likely well-known to specialists, though we were unable to find a reference.

\begin{qst}
Is there a commutative ring with unity for which the Borel subgroup $\Bzero$ of $\SL_2(R)$ is finitely presented, but the elementary subgroup $E_2(R) \leq \SL_2(R)$ itself is not? Equivalently, is the kernel of the natural map $St(A_1, R) \onto E_2(R)$ finitely generated as a normal subgroup whenever $\Bzero$ is finitely presented? Here, $St(\Phi, R)$ denotes the (unstable) Steinberg group of type $\Phi$ over $R$.
\end{qst}

\subsection{} Apart from restrictions on the characteristic, the missing piece for a complete characterization in the form of Theorem \ref{B} is the case where $\K$ is a global function field and $|S| = 1$. It can be shown that a maximal parabolic $\P(\F_q[t]) \leq \SL_n(\F_q[t])$ is finitely presented if and only if so is its Levi factor, and this is likely extendable to the other simply-laced root systems. Whether this holds for types $B, C, F$ and $G$ is unknown to us.

\printbibliography

 \end{document}